\documentclass[a4paper,11p]{amsart}

\usepackage[all,arc]{xy}
\usepackage[dvipsnames]{xcolor}
\usepackage{graphicx} 
\usepackage[centertags]{amsmath}
\usepackage{amscd}
\usepackage{amsthm}
\usepackage{amssymb}
\usepackage{dsfont} 
\usepackage{MnSymbol} 
\usepackage{algorithm}
\usepackage{algpseudocode}
\usepackage{tikz-cd}
\usetikzlibrary{cd}
\usepackage{tikz}
\usepackage{relsize}
\usepackage{ytableau}
\usepackage{mdframed}
\usepackage{comment}
\usepackage{enumitem}
\usepackage[utf8]{inputenc}
\usepackage[T1]{fontenc}

\usepackage{textcomp}

\usepackage{hyperref}
\hypersetup{
    colorlinks,
    linkcolor={MidnightBlue},
    citecolor={Green}
}
\usepackage[nameinlink]{cleveref}

\usepackage[normalem]{ulem}

\newcommand{\m}{\mathfrak m}
\newtheorem{notation}{Notation}[section]
\newtheorem{defn0}{Definition}[section]
\newtheorem{prop0}[defn0]{Proposition}
\newtheorem{thm0}[defn0]{Theorem}
\newtheorem{lemma0}[defn0]{Lemma}
\newtheorem{corollary0}[defn0]{Corollary}
\newtheorem{example0}[defn0]{Example}
\newtheorem{remark0}[defn0]{Remark}
\newtheorem{conjecture0}[defn0]{Conjecture}
\DeclareFontFamily{U}{MnSymbolC}{}
\DeclareSymbolFont{MnSyC}{U}{MnSymbolC}{m}{n}
\DeclareFontShape{U}{MnSymbolC}{m}{n}{
    <-6>  MnSymbolC5
   <6-7>  MnSymbolC6
   <7-8>  MnSymbolC7
   <8-9>  MnSymbolC8
   <9-10> MnSymbolC9
  <10-12> MnSymbolC10
  <12->   MnSymbolC12}{}
\DeclareMathSymbol{\aprod}{\mathbin}{MnSyC}{'270}

\newenvironment{definition}{ \begin{defn0}}{\end{defn0}}
\newenvironment{proposition}{\bigskip \begin{prop0}}{\end{prop0}}
\newenvironment{theorem}{\bigskip \begin{thm0}}{\end{thm0}}
\newenvironment{lemma}{\bigskip \begin{lemma0}}{\end{lemma0}}
\newenvironment{corollary}{\bigskip \begin{corollary0}}{\end{corollary0}}
\newenvironment{example}{ \begin{example0}\rm}{\end{example0}}
\newenvironment{remark}{ \begin{remark0}\rm}{\end{remark0}}

\newtheorem{alg}{Algorithm}
\newcommand{\K}{\mathbb K}

\title{A refinement on the local cactus rank algorithm}

\author{Alessandra Bernardi, Oriol Reig Fité}

\address{Universit\`a di Trento, Via Sommarive, 14 - 38123 Povo (Trento), Italy}
\email{alessandra.bernardi@unitn.it, oriol.reigfite@unitn.com}

\makeatletter
\@namedef{subjclassname@2020}{\textup{2020} Mathematics Subject Classification}
\makeatother


\begin{document}
\begin{abstract} We present an algorithm to  recover a minimal local apolar scheme to a homogeneous polynomial $F$.
The scheme’s socle degree determines whether it is evinced by a Generalized Additive Decomposition (GAD) of $F$ or of an extension.
We give constructive procedures for both cases and compute the Hilbert function efficiently via Hankel operators.
\end{abstract}
\maketitle
\section*{Introduction}

The study of ranks of homogeneous forms has a long history, starting from Sylvester’s classical theorem on binary forms (cf. eg. \cite{BERNARDI201134, Bernard}), and continuing through the modern theory of secant varieties and tensor ranks.  
Among these notions, the \emph{cactus rank} of a form, defined as the minimal length of a zero-dimensional apolar scheme (cf. \cite{IK, BR13}), plays a key role in algebraic geometry and its applications (cf. eg. \cite{MR3729273}).  
When the minimal apolar scheme to a given polynomial $F$ is local, its structure captures subtle information about the form and the possible additive decompositions of either $F$ itself or an extension of $F$.

A classical way to describe such forms is via \emph{Generalized Additive Decompositions} (GADs). The connection between local Gorenstein schemes and GADs has been explored in works such as \cite{IK,MR3250539,GAD, BJPR}, where criteria are given to decide when a local scheme is evinced by a GAD.  
In these settings, the Macaulay correspondence between Gorenstein algebras and divided power series is a standard tool, though the presentation is often given in the standard polynomial ring and the divided power formalism is introduced only when necessary.

Recent algorithmic approaches to cactus rank and GADs have focused on practical computation of apolar schemes \cite{Alessandra, GAD}.  
However, most of these algorithms are formulated in the standard polynomial setting, and the explicit transition to divided powers is typically performed only in the final stages.  
This choice, while standard in the literature, may introduce additional technicalities when performing coordinate changes, homogenizations, or computations involving higher socle degrees.

\medskip
\noindent
\textbf{Our contribution.}  
In this paper, we develop a complete algorithm to determine the local cactus rank of a form $F$ \emph{entirely in the divided power setting}, to recover its minimal local apolar scheme, and to decide whether the scheme is evinced by either a GAD of $F$ or of an extension of $F$.  
The algorithm is constructive: it computes the support of the minimal scheme from multiplication operators, extracts the Hilbert function and its symmetric decomposition, and recovers the GAD when possible.  
Our presentation is self-contained in the divided power setting, so that all coordinate changes and dual actions are explicit.

\medskip
\noindent
\textbf{Structure of the paper.}  
\Cref{sec:prelim} recalls the necessary background on divided powers, the dual space, Gorenstein algebras, and Hilbert functions, setting the notation used throughout the paper.  
Although these notions are classical, we present them in detail because the algorithms are implemented entirely in the divided power context, whereas the standard exposition is usually given in the polynomial ring.   
An expert reader may consult it mainly to verify the notation. 

\Cref{sec:main} recalls the background on apolarity, local Gorenstein algebras, and Generalized Additive Decompositions (GAD), and contains the main results. It includes the description of the local cactus rank algorithm, and the recovery of the local GAD in the two cases where the socle degree of the minimal local apolar scheme is either smaller or bigger than the degree of $F$, and
the computation of the Hilbert function via Hankel operators  with several examples and computational remarks.

Our main contributions are:
\begin{itemize}
    \item An explicit algorithm for computing local cactus rank and minimal local apolar schemes directly in the divided power setting.
    \item A constructive method to determine whether a minimal local apolar scheme is evinced by a GAD of a polynomial $F$, and to recover the decomposition when it exists and generalize it to an extension of $F$.
    \item A unified treatment of background material and algorithmic procedures in the divided power framework, avoiding repeated transitions from standard polynomials.
\end{itemize}

\section*{Acknowledgements}
We are grateful to Jarosław Buczyński, Maciej Gałązka, Joachim Jelisiejew, Daniele Taufer, and Bernard Mourrain for insightful conversations that greatly benefited this work.

\medskip

This work has been supported by European Union’s HORIZON–MSCA-2023-DN-JD programme under the Horizon Europe (HORIZON) Marie Skłodowska-Curie Actions, grant agreement 101120296 (TENORS).

\section{Preliminaries and Background}\label{sec:prelim}
In this section we revisit some foundational concepts and fix the notation used throughout the paper. Our aim is to provide a self-contained introduction to Gorenstein algebras and Hilbert functions in the divided power setting, making these essential tools accessible even to non-expert readers.


\medskip

    Let $\mathbb{K}$ be an algebraically closed field of characteristic zero, and let $V$ be a $\mathbb{K}$-vector space of dimension $n+1$. Denote by $S := 
    \mathbb{K}[x_0, \ldots, x_n]$ the symmetric algebra of $V$, identified with the homogeneous coordinate ring of projective space $\mathbb{P}^n$. Let also $R := \mathbb{K}[x_1, \ldots, x_n]$ be the coordinate ring of the standard affine chart $\{x_0 \neq 0\} \subset \mathbb{P}^n$.

\medskip

Denote by $S_d$ (respectively $R_d$) the space of homogeneous polynomials of degree $d$ in $S$ (respectively in $R$). Similarly, $S_{\leq d}$ and $R_{\leq d}$ denote polynomials of degree at most $d$. The dual space of a vector space $V$ is denoted by $V^*:=\mathrm{Hom}_\mathbb{K}(V, \mathbb{K})$.

\medskip

For multi-indices $\alpha = (\alpha_1, \ldots, \alpha_n)$ and $\beta = (\beta_1, \ldots, \beta_n)$ in $\mathbb{N}^n$, we use standard multi-index notations:
\[
\binom{d}{\alpha} := \frac{d!}{\alpha!(d - |\alpha|)!}, \quad \text{where} \quad \alpha! = \prod_{i=1}^n \alpha_i!, \quad |\alpha| = \sum_{i=1}^n \alpha_i,
\]

and similarly for $\binom{\alpha}{\beta}:= \frac{\alpha!}{\beta!(\alpha - \beta)!}$. 

Given $\zeta \in \mathbb{K}^n$, the evaluation at $\zeta$ is the functional $\mathds{1}_\zeta : R \to \mathbb{K}$, defined as $\mathds{1}_\zeta(p) = p(\zeta)$.

Finally, we abbreviate products of the form $(x_0 - \zeta_0)^{\alpha_0} \cdots (x_n - \zeta_n)^{\alpha_n}$ as $(x - \zeta)^\alpha$.

\subsection{Divided powers and the dual space}
The dual space $S^* := \mathrm{Hom}_{\mathbb{K}}(S, \mathbb{K})$ can be identified with the power series ring in the basis of divided powers. Consider the monomial basis ${x^\alpha}$ of $S$ and its dual basis ${Y^{(\alpha)}}$ in $S^*$, defined by $Y^{(\alpha)}(x^\beta)=\delta(\alpha, \beta)$. This yields an isomorphism
of $ \K $-vector spaces:
\[
\Psi: S^* \longrightarrow \K[[Y_0, \ldots, Y_n]], \qquad
\Lambda \mapsto \sum_{\alpha \in \mathbb{N}^{n+1}} \Lambda(x^\alpha)\, Y^{(\alpha)}.
\]
An affine translation 
\begin{equation}\label{affine:tranlsation}
\varphi: R \to R, \qquad \varphi(x_i) = x_i + \zeta_i
\end{equation}
induces a dual map $\varphi^*: R^* \to R^*$ acting as differential operators:
\begin{equation}\label{dual:affine:tranlsation}
    \varphi^*: R^* \longrightarrow R^*, \qquad \varphi^*(\Lambda)(p) := \Lambda(\varphi(p)).
\end{equation}

In particular,
$\varphi^*(1)=\sum_{\alpha\in \mathbb N^n} \varphi^*(1)(x^\alpha) Y^{(\alpha)}=\sum_{\alpha\in \mathbb N^n} \zeta^\alpha Y^{\alpha}= \mathds{1}_\zeta$ where $ \mathds{1}_\zeta $ is the evaluation at the point $ \zeta $.

\begin{lemma}\label{lemma:DualChangeBasis}
Let $ \K $ be a field of characteristic zero. Let $ \zeta \in \K^n $ and $ \varphi$ be the affine change of coordinates of \eqref{affine:tranlsation}. Then, for every $ \alpha \in \mathbb{N}^n $,
\[
\varphi^*\left(Y^{(\alpha)}\right) = \frac{1}{\alpha!} \mathds{1}_\zeta \circ \frac{\partial^\alpha}{\partial x^\alpha}.
\]
\end{lemma}

\begin{proof}
Due to linearity, we evaluate the expression on $x^\beta$:
\[
\varphi^*\left(Y^{(\alpha)}\right)\left(x^\beta\right) = Y^{(\alpha)}\left((x + \zeta)^\beta\right).
\]
By the binomial theorem, $(x + \zeta)^\beta = \sum_{k \le \beta} \binom{\beta}{k} x^k \zeta^{\beta - k}$. Applying $Y^{(\alpha)}$ extracts the coefficient of $x^\alpha$, which is:
\[
\prod_{i=1}^n \binom{\beta_i}{\alpha_i} \zeta_i^{\beta_i - \alpha_i} \quad \text{for } \alpha \leq \beta, \text{ and } 0 \text{ otherwise}.
\]
This is equivalent to $\frac{1}{\alpha!} \frac{\beta!}{(\beta - \alpha)!} \zeta^{\beta - \alpha}$ (when $\alpha \le \beta$), precisely matching $\left(\frac{1}{\alpha!} \mathds{1}_\zeta \circ \frac{\partial^\alpha}{\partial x^\alpha}\right)\left(x^\beta\right)$.
\end{proof}

From this, we deduce:
\[
\left( \mathds{1}_\zeta \circ \frac{\partial^\alpha}{\partial x^\alpha} \right) \left( \frac{1}{\alpha!} (x - \zeta)^\beta \right) =
\begin{cases}
1 & \text{if } \alpha = \beta, \\
0 & \text{otherwise}.
\end{cases}
\]
This implies that $\mathcal{B}_\zeta := \left\{ \frac{1}{\alpha!}(x - \zeta)^\alpha \right\}_{\alpha \in \mathbb{N}^n}$ forms a basis for $R$. Its dual basis is given by $ \{\partial^\alpha_\zeta\} \subset R^* $, where $ \partial^\alpha_\zeta(p) := \left( \mathds{1}_\zeta \circ \frac{\partial^\alpha}{\partial x^\alpha} \right)(p) $, and $\partial^{\alpha}_{\zeta}$ acts naturally as $\left(\mathds{1}_{\zeta} \circ x^{\alpha}(\partial)\right)(p)$. This establishes an isomorphism $S^* \to \mathbb{K}[[\partial_{1, \zeta},\ldots,\partial_{n, \zeta}]]$ mapping $\Lambda \mapsto \sum_{\alpha\in \mathbb N^{n+1}} \partial^{\alpha}_{\zeta}(x^{\alpha}) \partial^{\alpha}_{\zeta}$. Note that this demonstrates the dual change of basis from $\left\{Y^{(\alpha)}\right\}$ to $\mathcal{B}_\zeta$'s dual basis (see \Cref{lemma:DualChangeBasis}). This formalism also extends to positive characteristic fields $\mathbb{K}$ by replacing $\mathcal{B}_\zeta$ with $\left\{(x-\zeta)^\alpha\right\}_{\alpha \in \mathbb{N}^n}$.

\subsubsection{Comultiplication and divided powers}

The diagonal map $ v \mapsto (v,v) $ on $ V $ induces a comultiplication on $ S $, defined by:
\[
\Delta(x_i) = x_i \otimes 1 + 1 \otimes x_i, \qquad
\Delta(x^\alpha) = \sum_{\beta + \gamma = \alpha} \binom{\alpha}{\beta} x^\beta \otimes x^\gamma.
\]
The multiplication on $ S^* $ is defined by:
\[
(\Lambda_1 \cdot \Lambda_2)(p) := (\Lambda_1 \otimes \Lambda_2)(\Delta(p)).
\]
In this setting:
\[
Y^{(\alpha)} \cdot Y^{(\beta)} = \binom{\alpha + \beta}{\alpha} Y^{(\alpha + \beta)}, \qquad
Y_0^{\alpha_0} \cdots Y_n^{\alpha_n} = \alpha! Y^{(\alpha)}.
\]
One also defines, for any $ F \in S^* $ and $ k \geq 1 $,
$F^{(k)} := \frac{1}{k!} F^k$.

The ring $\mathbb{K}_{dp}[Y_0,\ldots,Y_n]$ is called the \emph{divided power ring}, and it is isomorphic to a polynomial ring in characteristic zero via $Y^{(\alpha)} \mapsto \frac{1}{\alpha!} Z^\alpha$.





\subsubsection{Grading and homogeneous coordinate changes}

The decomposition $ S = \bigoplus_{d \geq 0} S_d $ induces a grading on the divided power ring, where:
\[
\mathrm{Hom}_{\mathbb{K}}(S_d, \mathbb{K})* \cong \K_{dp}[Y_0,\ldots,Y_n]_d, \qquad
\K_{dp}[Y_0,\ldots,Y_n] = \bigoplus_{d \geq 0} \mathrm{Hom}_{\mathbb{K}}(S_d, \mathbb{K}).
\]
Note that $S^*=\prod_d \mathrm{Hom}_{\mathbb{K}}(S_d, \mathbb{K})$. Now take $ \zeta = (\zeta_1, \ldots, \zeta_n) \in \K^n $, and define a homogeneous change of coordinates:

\begin{equation}\label{change:of:coordinates}
\varphi(x_0) = x_0, \qquad \varphi(x_i) = x_i + \zeta_i x_0 \quad \text{for } i = 1, \ldots, n.
\end{equation}

The dual map $ \varphi^*: S^* \to S^* $ preserves degrees, and acts on each graded piece $ S_d^* \cong \K_{dp}[Y_0,\ldots,Y_n]_d $.

\begin{example}\label{example: DualChangeCoord}
Let $ n = 1 $, $ \zeta = 2 $. Then the map $ \varphi: \K[x_0,x_1]_2 \to \K[x_0,x_1]_2 $ has matrix:
\[
\begin{bmatrix}
1 & 2 & 4 \\
0 & 1 & 4 \\
0 & 0 & 1
\end{bmatrix}
\]
in the monomial basis. Its transpose gives the matrix of $ \varphi^*: \K_{dp}[Y_0,Y_1]_2 \to \K_{dp}[Y_0,Y_1]_2 $.
\end{example}


Furthermore, due to the functoriality of the symmetric algebra, a homogeneous change of coordinates $\varphi$ in $S$ induces an algebra homomorphism 
\begin{equation}\label{prop: DualChangeIsAlgebraHom}
    \varphi^*: \mathbb{K}_{dp}[Y_0,\ldots,Y_n] \to \mathbb{K}_{dp}[Y_0,\ldots,Y_n]
    \end{equation}
on the divided power algebra (cf. \cite[Chap. III, \S11.5]{Bourbaki}). In particular, we can use the multiplicity of changes of coordinates to avoid computing the matrices in \Cref{example: DualChangeCoord}.
\begin{proposition}\label{prop: ChangeCoordInDivPowers}
    Let $\zeta=(\zeta_1,\ldots, \zeta_n)\in \mathbb K^n$ be an affine point, $\varphi:S\to S$ the homogeneous change of coordinated sending $[1:0:\ldots :0]$ to $[1:\zeta_1:\ldots : \zeta_n]$. Let $\Psi$ be the algebra ismomorphism converting divided powers expressions into polynomials. Then the algebra homomorphism $\Tilde{\varphi}:S\to S$ induced by $x_0\mapsto x_0+\zeta_1x_1+\ldots + \zeta_n x_n$, $x_i\mapsto x_i$ makes the following diagram commute:

\[\begin{tikzcd}
	{K_{dp}[Y_0,\ldots, Y_n]} & {K_{dp}[Y_0,\ldots, Y_n]} \\
	S & S
	\arrow["{\varphi^*}", from=1-1, to=1-2]
	\arrow["\Psi", from=1-1, to=2-1]
	\arrow["\Psi", from=1-2, to=2-2]
	\arrow["{\Tilde{\varphi}}", from=2-1, to=2-2]
\end{tikzcd}\]    
\end{proposition}

\subsection{Macaulay correspondence and Gorenstein algebras}
Define the contraction action of $S$ on $S^*$ as the dual operation to multiplication:
\[
(p \aprod \Lambda)(q) = \Lambda(pq), \quad p,q \in S,\; \Lambda \in S^*.
\]

With the choice of coordinates and dual basis introduced above, one can verify that on monomials:

\begin{equation}\label{eq: ContractionOnMonomials}
x^{\alpha} \aprod Y^{(\beta)} =
\begin{cases}
Y^{(\beta - \alpha)} & \text{if } \alpha_i \leq \beta_i \text{ for all } i, \\
0 & \text{otherwise}
\end{cases}.
\end{equation}

Thus, on polynomials in $\K_{dp}[Y_0, \ldots, Y_n]$, contraction acts as (scaled) partial derivation. In fact, contraction and derivation induce two isomorphic $S$-module structures on $\K_{dp}[Y_0, \ldots, Y_n]$ (see \cite[Proposition 3.13]{Joachim}).

\bigskip

We also consider the contraction action of $R = \K[x_1, \ldots, x_n]$ on $R^*$. As seen in \eqref{eq: ContractionOnMonomials}, contraction of two monomials lowers the degree. This is not necessarily true for infinite series in $\K_{dp}[[Y_1, \ldots, Y_n]]$. For example, for any $p \in R$:
\[
p \aprod \mathds{1}_{\zeta} = p(\zeta)\mathds{1}_\zeta,
\]
since for any $q \in R$,
$
(p \aprod \mathds{1}_\zeta)(q) = \mathds{1}_\zeta(pq) = \mathds{1}_\zeta(p) \cdot \mathds{1}_\zeta(q) = p(\zeta) \mathds{1}_\zeta(q).
$

\begin{remark}\label{remark: contrVsorth}
    For all $p \in S$ and $\Lambda \in S^*$, if $p \aprod \Lambda = 0$, then $\Lambda(p) = 0$. Indeed,
    \[
    \Lambda(p) = \Lambda(p \cdot 1) = (p \aprod \Lambda)(1) = 0.
    \]
    Clearly the same holds for $\Lambda \in R^*$ and $p \in R$. However, the converse does not hold: for instance, $Y^{(3)}(x^2) = 0$ (since $Y^{(3)}$ is the dual of $x^3$), but $x^2 \aprod Y^{(3)} = Y$.
\end{remark}

\begin{definition}
    Let $\Lambda \in S^*$, where $S^*$ is viewed as an $S$-module via contraction. We define the \emph{annihilator} of $\Lambda$ as
    \[
    \operatorname{Ann}(\Lambda) := \{ p \in S \mid p \aprod \Lambda = 0 \}.
    \]
    An ideal $I \subseteq S$ is said to be \emph{apolar} to a homogeneous polynomial $F \in \K_{dp}[Y_0, \ldots, Y_n]_d$ if $I \subseteq \operatorname{Ann}(F)$. A zero-dimensional scheme $Z$ is apolar to $F$ if $I(Z)$ is apolar to $F$.
\end{definition}

Note that the definition of contraction is basis-independent, so $\operatorname{Ann}(\Lambda)$ does not depend on the specific representation of $\Lambda$.

The classical Apolarity lemma can be stated as follows.
\begin{lemma}[Apolarity Lemma]\label{apolarity:lemma}
  An ideal $I\subset S$ is apolar to   $F\in S_d^*$ if and only if
%
    $I_d\subseteq \{p\in S_d\; | \; p\aprod F =0\}$.
\end{lemma}
After having introduced the standard dehomogenization isomorphism $\pi: S_d\cong R_{\leq d}$ and its dual defined by: $(\pi^{-1})^*:(S_d)^*\to (R_{\leq d})^*$, s.t.  $  f^*(p):=F(\pi^{-1}(p))$,  for $f^*=F(Y_0=1)$,
we can also state
an affine version of the apolarity criterion. 

\begin{lemma}[Affine Apolarity Criterion]\label{affine:apolarity:criterion}
For $F\in S_d^*$ homogeneous and $f=F(Y_0=1)$, we have:
\begin{enumerate}[label=\roman*),leftmargin=*]
\item   A zero-dimensional scheme $Z\subseteq \mathbb{P}^n$ defined locally by $\operatorname{Spec}(R/I)$ is apolar to $F$ if and only if $f(p)=0$ for every $p\in I$ with degree $\leq d$.
\item Any $\Lambda\in R^*$ extending $f$ defines a scheme $Z=\operatorname{Spec}(R/\ker(H_\Lambda))$ apolar to $F$.
\end{enumerate}
\end{lemma}

\begin{proof}
\begin{enumerate}[label=\roman*),leftmargin=*]
    \item\label{rmk: affineApolarity} \leavevmode
  By \Cref{apolarity:lemma}, $I(Z)\subseteq \operatorname{Ann}(F)$ if and only if $I(Z)_d\subseteq \operatorname{Ann}(F)$. Since contraction and the duality pairing agree on $S_d$, this is equivalent to say that $F(I(Z)_d)=0$, i.e. for every $p\in I\subseteq R$ of degree $\leq d$, $F(\pi^{-1}(p))=0$. This is to say that $(\pi^{-1})^*(F)(p)=0$ for every $p\in I$ of degree $\leq d$, or equivalently $F(Y_0=1)(p)=0$ for every $p\in I$ of degree $\leq d$.
\item\label{remark:extensionsAreApolar} If $Z$ is locally defined by $\ker(H_\Lambda)$ then by what we just proved in the previous item, $I(Z)\subseteq \operatorname{Ann}(F)$ if and only if $f(p)=0$ for all $p\in \ker(H_\Lambda)\subseteq R$ with $p\in R_{\leq d}$. And the last condition holds since $p\in \ker(H_\Lambda)$ implies $\Lambda(p)=0$ and $\Lambda$ extends $f$.
\end{enumerate}
\end{proof}

We further generalize this result explicitly for schemes supported at arbitrary points:

\begin{lemma}
Let $F\in S_d^*$ homogeneous and $Z$ a local Gorenstein scheme supported at $[1:\zeta_1:\dots:\zeta_n]$ with $I(Z)\subseteq \operatorname{Ann}(F)$. Then there exists $\Lambda=\mathds{1}_\zeta\circ H(\partial)$, $H\in R$, such that $Z$ is defined by $\operatorname{Ann}(\Lambda)$ and the form $\Lambda$ coincides with $f$ in degree $\leq d$. In particular, $f$ is the truncation of the infinite series $\Lambda$ in degree $d$. 
\end{lemma}

\begin{proof}
    Let $\phi$ be a change of coordinates in $S$ that maps the point $[1 : \zeta_1 : \dots : \zeta_n]$ to the origin $[1 : 0 : \dots : 0]$. This transformation is induced by the linear map on coordinates $x_0 \mapsto x_0$ and $x_i \mapsto x_i - \zeta_i x_0$ for $i=1, \dots, n$.

This projective transformation $\phi$ induces an affine change of coordinates $\phi'$ on the space of polynomials $R_{\le d}$ (of degree at most $d$), given by $x_i \mapsto x_i - \zeta_i$.

Under these transformations:
\begin{itemize}
    \item The form $F \in S_d^*$ is transformed into $\tilde{F} = \phi^*(F) \in S_d^*$.
    \item The associated affine form $f^*$ is transformed into $\tilde{f}^* = (\pi^{-1})^* \circ \phi^*(F) = (\phi')^* \circ (\pi^{-1})^*(F) \in R_{\le d}$.
    \item The support $Z$ of the ideal is now located at the origin $[1 : 0 : \dots : 0]$.
\end{itemize}

With $Z$ supported at the origin, the hypotheses of \cite[Proof of Proposition 4]{BJPR} are now satisfied. Therefore, the ideal $I(Z)$ is locally defined by $\text{Ann}(g)$, where $g \in K_d[Y_1, \dots, Y_n] \subseteq R^*$ is a polynomial that defines the same form as $\tilde{f}^*$ in $R_{\le d}$.

Now, we reverse the change of coordinates. According to the change of basis \Cref{lemma:DualChangeBasis}, the ideal $I(Z)$ is supported at its original point $[1 : \zeta_1 : \dots : \zeta_n]$. Locally, it is defined by $\text{Ann}(\Lambda)$, where $\Lambda$ is of the form $1_\zeta \circ H(\partial)$ for some polynomial $H \in R$ such that $\Lambda$ extends the form $f$.

Finally, we note that the homogenization of $\text{ker}(H_\Lambda)$ is apolar to $F$ by \Cref{rmk: affineApolarity}.
\end{proof}

\subsubsection{Annihilators of infinite series}

Since evaluation maps $\mathds{1}_\zeta$ are defined only on the affine space, we restrict our discussion to the affine coordinate ring $R$ and its dual $R^*$.

\begin{remark}
    For every polynomial $f \in \K_{dp}[Y_1, \ldots, Y_n]$, one has
    \[
    (x_1, \ldots, x_n)^{\deg f + 1} \subseteq \operatorname{Ann}(f) \subseteq (x_1, \ldots, x_n),
    \]
    hence $R/\operatorname{Ann}(f)$ is a local ring supported at $0$. This need not hold for annihilators of infinite series. For instance, let $\zeta_1, \zeta_2 \in \K^n$ and define $\Lambda := \mathds{1}_{\zeta_1} + \mathds{1}_{\zeta_2} \in R^*$. Then $\mathcal{V}(\operatorname{Ann}(\Lambda)) = \{\zeta_1, \zeta_2\}$, where $\mathcal{V}(I)$ denotes the variety defined by the ideal $I$.

    Indeed, if $p$ vanishes at both $\zeta_1$ and $\zeta_2$, then for all $q \in R$:
    \[
    (p \aprod \Lambda)(q) = \mathds{1}_{\zeta_1}(p)\mathds{1}_{\zeta_1}(q) + \mathds{1}_{\zeta_2}(p)\mathds{1}_{\zeta_2}(q) = 0,
    \]
    so $p \in \operatorname{Ann}(\Lambda)$. Conversely, let $p \in \operatorname{Ann}(\Lambda)$ and let $q \in R$ such that $q(\zeta_1) \neq 0$, $q(\zeta_2) = 0$, then
    \[
    0 = (p \aprod \Lambda)(q) = \mathds{1}_{\zeta_1}(p) \cdot \mathds{1}_{\zeta_1}(q),
    \]
    forcing $p(\zeta_1) = 0$. Repeating the argument for $\zeta_2$ shows that $p$ vanishes at both points.
\end{remark}

By linearity of contraction, the annihilator of $\Lambda \in R^*$ is the kernel of the linear map known as the \emph{Hankel operator} associated to $\Lambda$:
\[
H_{\Lambda} : R \to R^*, \quad p \mapsto p \aprod \Lambda.
\]

\begin{lemma}\label{rmk: changeCoordinatesKernel}
    Let $\zeta \in \K^n$ and let $\varphi$ 
    be the affine change of coordinates \eqref{affine:tranlsation}. 
 Then for any $\Lambda \in R^*$:
    \[
    \varphi(\ker H_\Lambda) = \ker H_{(\varphi^{-1})^*(\Lambda)},
    \]
    where $\varphi^*$ denotes the dual change of coordinates \eqref{dual:affine:tranlsation} given by the transpose of $\varphi$.
\end{lemma}

\begin{proof}
    We show both inclusions. Let $\varphi(p) \in \varphi(\ker H_\Lambda)$. Then for all $q \in R$:
    \[
    (\varphi(p) \aprod (\varphi^{-1})^* \Lambda)(q) = \Lambda\left( (\varphi^{-1} \circ \varphi)(p) \cdot \varphi^{-1}(q) \right) = (p \aprod \Lambda)(\varphi^{-1}(q)) = 0.
    \]
    Conversely, if $p \aprod (\varphi^{-1})^* \Lambda = 0$, then for all $q \in R$:
    \[
    (\varphi^{-1}(p) \aprod \Lambda)(q) = \Lambda\left( \varphi^{-1}(p) \cdot \varphi^{-1}(\varphi(q)) \right) = ((\varphi^{-1})^* \Lambda)(p \cdot \varphi(q)) = (p \aprod (\varphi^{-1})^* \Lambda)(\varphi(q)) = 0.
    \]
\end{proof}
\begin{definition}
    Let $A$ be an Artinian $\K$-algebra. Its \emph{canonical module} $\omega_A$ is the dual vector space $\operatorname{Hom}_\K(A, \K)$ with $A$-module structure given by contraction:
    \[
    (p \aprod \Lambda)(q) = \Lambda(pq) \quad \forall p,q \in A,\, \Lambda \in \omega_A.
    \]
    We say that $A$ is \emph{Gorenstein} if $\omega_A$ is isomorphic to $A$ as an $A$-module, i.e., it is generated by a single element (called the \emph{dual generator}).
\end{definition}

Gorenstein Artinian algebras are quotients of Henkel operators with finite rank.

\begin{lemma}[{\cite[\S 3.1]{Polyexp}}]\label{lemma: GorensteinHankel}
    Let $\Lambda \in R^*$ be such that the algebra $R / \ker(H_\Lambda)$ has finite dimension. Then $R / \ker(H_\Lambda)$ is Gorenstein, with dual generator $\Lambda$. Conversely, every Gorenstein Artinian algebra is isomorphic to $R / \ker(H_\Lambda)$ for some $\Lambda \in R^*$.
\end{lemma}

Gorenstein Artinian algebras are typically studied in the local case (supported at $0$), since any Artinian algebra decomposes as a sum of local Artinian algebras and Gorensteinness is a local property. Recall that for an Artinian local ring $(A, \mathfrak{m})$, the largest integer $d$ such that $\mathfrak{m}^d \neq 0$ is called the \emph{socle degree} of $A$.

\begin{theorem}[{\cite[Lemma 1.2]{Iarrobino}}]
    Let $A = R/I$ be a Gorenstein Artinian algebra supported at $0 \in \K^n$ with socle degree $d$. Then $A = R/\ker(H_\Lambda)$ for some polynomial $\Lambda \in \K_{dp}[Y_1, \ldots, Y_n]$ of degree $d$. Conversely, if $\Lambda \in \K_{dp}[Y_1, \ldots, Y_n]$ is a polynomial of degree $d$, then $R / \ker(H_\Lambda)$ is a local Gorenstein algebra supported at $0$, with socle degree $d$. If $A$ is graded, $\Lambda$ can be chosen homogeneous.
\end{theorem}

It is possible to establish a version of this result for arbitrary coordinates: apply  \Cref{rmk: changeCoordinatesKernel}, \Cref{lemma:DualChangeBasis} and obtain the corresponding change of basis in $R^*$.

\begin{corollary}[Generalized Macaulay Correspondence]\label{Thm: GeneralisedMacaulayCorrespondence}
    Let $A$ be a Gorenstein Artinian algebra of socle degree $d$ supported at $\zeta \in \K^n$. Then $A \cong R / \ker(H_\Lambda)$ for some $\Lambda \in R^*$ of the form $\mathds{1}_\zeta \circ H(\partial)$, where $H \in R$ is a polynomial of degree $d$ and $H(\partial)$ denotes the differential operator obtained by replacing each $x_i$ with $\partial / \partial x_i$. Conversely, if $\Lambda = \mathds{1}_\zeta \circ H(\partial)$ with $\deg H = d$, then $R / \ker(H_\Lambda)$ is a local Gorenstein algebra supported at $\zeta$, of socle degree $d$.
\end{corollary}

\subsection{Hilbert Functions of Gorenstein algebras}

The Hilbert function of the associated graded algebra of a local Gorenstein algebra must satisfy certain structural properties, and admits a symmetric decomposition induced by the interplay of two natural filtrations. We recall the basic constructions from \cite{Iarrobino}.

\begin{definition}
    Let $A$ be a local Artinian $\K$-algebra with maximal ideal $\mathfrak{m}$. Its \emph{associated graded algebra} is
    \[
    gr(A) := \bigoplus_{i \geq 0} \mathfrak{m}^{i} / \mathfrak{m}^{i+1}.
    \]
    The \emph{Hilbert function} of $A$ is defined as
    \[
    H_A(i) := \dim_{\K} \left( \mathfrak{m}^{i} / \mathfrak{m}^{i+1} \right).
    \]
\end{definition}

Since $A / \mathfrak{m} \cong \K$, we have $H_A(0) = 1$. The \emph{socle degree} of $A$ is the largest integer $i$ for which $H_A(i) \neq 0$.

As explained in \cite{Iarrobino}, if $A$ is graded, the dual generator of the Gorenstein algebra induces nondegenerate bilinear pairings between graded components $
A_i \times A_{d-i} \to \K$,
where $d$ is the socle degree. This implies that the Hilbert function is symmetric 
$H_A(i) = H_A(d - i)$, 
and in particular, $H_A(d) = 1$. Moreover, in the graded case one has
\[
(0 : \mathfrak{m}^{d - i}) = \mathfrak{m}^{i + 1}.
\]
However, this identity does not hold in general if $A$ is not graded (see e.g., \cite[Example 3.28]{Joachim}). In such cases, two canonical filtrations naturally arise:

\begin{itemize}
    \item the \emph{powers filtration}:
    \[
    \{0\} \subseteq \mathfrak{m}^d \subseteq \mathfrak{m}^{d-1} \subseteq \cdots \subseteq \mathfrak{m} \subseteq A;
    \]
    \item the \emph{Löwy filtration} (see \cite{Iarrobino}):
    \[
    \{0\} \subseteq (0 : \mathfrak{m}) \subseteq (0 : \mathfrak{m}^2) \subseteq \cdots \subseteq (0 : \mathfrak{m}^{d+1}) = A.
    \]
\end{itemize}

These filtrations are related through duality.

\begin{lemma}[{\cite[Lemma 2.31]{Joachim}}] 
    Let $A$ be a local Artinian $\K$-algebra, and let $I \subseteq A$ be an ideal. Let $\Lambda$ be a generator of the canonical module $\omega_A$, and consider the pairing
$    A \times A \to \K, \quad (p,q) \mapsto \Lambda(pq)$.
    Then:
    \[
    (0 : I) = \{ a \in A \mid \Lambda(ai) = 0 \text{ for all } i \in I \}.
    \]
\end{lemma}

In particular, there is a natural duality between the graded pieces of the two filtrations:
\[
\left(\frac{\mathfrak{m}^{i}}{\mathfrak{m}^{i+1}}\right)^* \cong \frac{(0 : \mathfrak{m}^{i+1})}{(0 : \mathfrak{m}^i)},
\]
and therefore:
\[
H_A(i) = \dim_\K \left( \frac{(0 : \mathfrak{m}^{i+1})}{(0 : \mathfrak{m}^i)} \right).
\]

This leads to the following family of subquotients: for each $a = 0, \ldots, d$, define
\[
C_{a,i} := \frac{(0 : \mathfrak{m}^{d+1 - a - i}) \cap \mathfrak{m}^i}{(0 : \mathfrak{m}^{d+1 - a - i}) \cap \mathfrak{m}^{i+1}},
\]
\[
C_a := \bigoplus_{i = 0}^{d - a} C_{a,i}, \quad Q_a := C_a / C_{a+1}.
\]
Let $\Delta_{Q_a}(i) := \dim_\K (Q_a)_i$ denote the Hilbert function of $Q_a$. From the definitions, one sees that:
\[
\Delta_{Q_0}(0) = 1, \quad \Delta_{Q_a}(0) = 0 \text{ for all } a > 0.
\]

The modules $Q_a$ give a symmetric decomposition of the Hilbert function of $A$:
\begin{theorem}[{\cite{Iarrobino}}]\label{thm:HF:decomposition}
    The Hilbert functions $\Delta_{Q_a}$ satisfy the following properties:
    \begin{enumerate}
        \item \textbf{Decomposition:}
        \[
        H_A(i) = \sum_{a = 0}^{d - i} \Delta_{Q_a}(i).
        \]
        
        \item \textbf{Symmetry:} Each $\Delta_{Q_a}$ is symmetric around $(d - a)/2$:
        \[
        \Delta_{Q_a}(i) = \Delta_{Q_a}(d - a - i).
        \]
        
        \item \label{HF: PartialSums} \textbf{Partial sums:} For each $k \geq 0$, the partial sum $\sum_{a = 0}^{k} \Delta_{Q_a}$ is the Hilbert function of a graded $\K$-algebra generated in degree 1.
    \end{enumerate}
\end{theorem}

\noindent
Note that \Cref{HF: PartialSums} in \Cref{thm:HF:decomposition} implies that each partial sum satisfies the Macaulay bound for Hilbert functions (see \cite{Joachim}).

\section{Refining the Cactus Rank Algorithm in Local Settings}\label{sec:main}
This section delves into the core contribution of our work: refining the classic algorithm for computing the cactus rank of a homogeneous polynomial (cf. \cite{Alessandra}) in a local setting. By leveraging the unique properties of local Gorenstein algebras, we develop effective methods to compute key algebraic invariants such as  the Hilbert function and its symmetric decomposition, within the algorithm's framework. Ultimately, these advancements pave the way for a more efficient and insightful computation of the generalized additive decomposition of a polynomial or of its extension, directly revealing the polynomial's cactus rank.
\subsection{The cactus algorithm}\label{sec:algorithms}

In this section, we recall how to compute a minimal apolar schemes to homogeneous polynomials in the global setting (cf. \cite{Alessandra}).

\begin{definition}[\cite{BR13}]
For $F\in S_d^*$, the \emph{cactus rank} of $F$ is the minimal length of a scheme apolar to $F$. The \emph{local cactus rank} is the minimal length of a local apolar scheme.
\end{definition}

The strategy developed in \cite{Alessandra} for the global cactus rank algorithm relies on the Hankel operator $H_\Lambda: R \to R^*$ defined as $H_\Lambda(p)(q)=\Lambda(pq)$. We recall here the crucial algebraic conditions from \cite{Bernard} for extending a truncated linear form $f\in R^*_{\leq d}$ to a form $\Lambda\in R^*$, necessary for defining finite-dimensional local algebras.

\begin{definition}
Let $V^* \subset R^*$ be two spaces equipped with a restriction map
\[
\mathrm{res} : R^* \longrightarrow V^*.
\]
Given $\Lambda \in V^*$, we say that $\tilde \Lambda \in R^*$ \emph{extends} $\Lambda$ if
\[
\mathrm{res}(\tilde \Lambda) = \tilde \Lambda\big|_{V^*} = \Lambda.
\]
\end{definition}

\begin{theorem}[\cite{Bernard}]
Given a monomial basis $B$ connected to $1$ and a truncated linear form $\Lambda\in \langle B\cdot B^+\rangle_{\leq d}^*$, an extension $\Tilde{\Lambda}\in R^*$ exists with $rank(H_{\Tilde{\Lambda}})=|B|$ and $B$ basis of $R/\ker(H_{\Tilde{\Lambda}})$ if and only if the associated matrices of multiplication by the variables commute and $\det(H_{\Lambda}^B)\neq 0$. Such extension is unique.
\end{theorem}


The essential ingredients used in \cite{Alessandra} for the cactus algorithm based on the previous result are:

\begin{itemize}
    \item (Cf.\cite{Bernard}) The fact that the multiplication operators $H_{a\thinstar \Lambda}$ for any $\Lambda\in R^*$ such that $\operatorname{rank} \; H_{\Lambda}<\infty$ and any $a\in A_{\Lambda}$ can be computed as
    \begin{equation}\label{MatMult}
    H_{a\thinstar \Lambda}=M_a^t \circ H_{\Lambda};\end{equation} 

    \item 
 (Cf. {\cite[Theorem 4.23]{BernardBook}}) \label{Thm: eigenvaluesMultOperatos}
        For every $a\in A_{\Lambda}$, the eigenvalues of the operators $M_a$ and $M_a^t$ are $\{a(\zeta_i), \ldots, a(\zeta_r)\}$ and the common eigenvectors of $M_{x_i}^t$, $i=1,\ldots,n$ are, up to scalar, $\mathds{1}_{\zeta_i}$.
    \item (Cf. {\cite[Proposition 3.2]{Alessandra}})\label{Prop: BasisInHankel}
    If $B\subseteq R/\ker H_\Lambda$ with $|B|=r$, then $B$ is a basis of the vector space $A_{\Lambda}$ if and only if $H_{\Lambda}^B$ is invertible and $\operatorname{rank} \; H_{\Lambda}=r$. Moreover, there exists a monomial basis $B$ that is a complete staircase.
\end{itemize}

\begin{remark}
The initial setup in \cite{Alessandra} for connecting a homogeneous polynomial $F \in S_d$ to a dual form in $S_d^*$ utilizes the apolar product introduced in \cite{Bernard}. Our approach, however, is intended to show that this apolar product can be entirely avoided by adopting the divided power setting, yielding the same dual form in $S_d^*$ in a more coordinate-free manner.

The apolar product of two homogeneous polynomials $F=\sum_{|\alpha|=d} f_{\alpha}x^{\alpha}$ and $G=\sum_{|\alpha|=d} g_{\alpha}x^{\alpha}\in S_d$, is: 
\[\langle G,F \rangle:=\sum_{|\alpha|=d} \frac{f_\alpha g_\alpha}{\binom{d}{\alpha}}=\frac{1}{d!}\sum_{|\alpha|=d} \alpha! f_\alpha g_\alpha. \]
In our divided power formalism, for $P\in S_d$ and $F\in S_d^*$, we use the apolarity pairing $\langle p, F\rangle_{\aprod}:= p\aprod F$ instead. If we identify the divided powers form of $F$ with a polynomial in $\K_{dp}[Y_0,...,Y_n]$, $F=\sum_{|\alpha|=d}\alpha! f_\alpha Y^{(\alpha)}$, then the pairing with $G=\sum_{|\alpha|=d} g_{\alpha}Y^{(\alpha)}$ is:
\[\langle G,F\rangle_\aprod=\langle G,\sum_{|\alpha|=d} \alpha! f_\alpha Y^{(\alpha)}\rangle_\aprod =\sum_{|\alpha|=d} \alpha! f_\alpha g_\alpha.\]
Evidently, these two expressions agree up to a factor of $d!$.
\end{remark}


\begin{notation}\label{notation: InitialRank}
Let $H_f^{\ostar}$ be a largest numerical submatrix of $H_{\Lambda(h)}$  with full rank.
\end{notation}
We can present the algorithm for the cactus rank described in \cite{Alessandra}.

\vspace{0.5cm}
\begin{mdframed}[]
\begin{alg}[Cactus rank and decomposition]\label{alg:cactus}\end{alg}
\vspace{0.1cm}
\noindent\textbf{Input:} A degree $d \geq 2$ polynomial $F \in S_d$ \\
    \textbf{Output:} Cactus rank of $F$ and its support $\zeta_1, \ldots ,\zeta_s\in \K^n$.

\begin{enumerate}
    \item Construct the matrix $H_{\Lambda(h)}$ with parameters $\{h_{\alpha}\}_{\alpha\in \mathbb N^n}$, $|\alpha|>d$.
    \item\label{alg:H:star} Set $r=\text{rank } H_f^{\ostar}$.
    \item\label{alg:basis} Take $B\subseteq R$ a complete staircase of monomials with $|B|=r$, do: 
    \begin{itemize}
        \item\label{alg:h's} Find $h$'s such that:
        \begin{itemize}
            \item $H_{\Lambda(h)}^B$ has nonzero determinant
            \item The multiplication operators $(M_{x_i})^t$ commute for all $i=1,\ldots,n$.
        \end{itemize}
        \item If found, go to \Cref{solutions}. If not, go to \Cref{alg:basis} with another choice of bases $B$. If all choices of $B$ with $|B|=r$ have been already performed, go to \Cref{increase:r}.
    \end{itemize}
    \item\label{increase:r} Set $r\to r+1$ and go to \Cref{alg:basis}.
    \item\label{solutions} The supports $\zeta_j$ of a minimal apolar scheme to $F$ are the common eigenvectors of $M_{x_i}$. The cactus rank of $F$ is the sum of the dimensions of the intersections over $j$ of the generalized egenspaces of $(M_{x_j}^B)^t$ relatives to $(\zeta_i)_j$. 
\end{enumerate}
\end{mdframed}
\vspace{0.5cm}

\subsection{The local cactus algorithm}\label{sec:GAD}

We now explore how the global cactus algorithm in \cite{Alessandra} can be adapted explicitly to the local setting. We start by understanding how to use the basic facts on Hilbert functions of local Artinian algebras described in \cite{Iarrobino}.

\subsubsection{Basis from Hilbert function and local conditions on parameters}

The Hilbert function (HF) of a graded Artinian Gorenstein algebra $R/I$, with $I$ homogeneous, uniquely determines the degrees of monomials in a monomial basis. However, for local Artinian Gorenstein algebras, if $I$ is not homogeneous, the Hilbert function alone does not determine a unique monomial basis. For example, in the algebra $\K[x,y]/(x-y^2,y^3)$, the Hilbert function is $(1,1,1)$, but distinct monomial bases  are possible: $\{1,x,y\}$  can be found via grevlex order, but since $x=y^2 \mod I$,  also $\{1,y,y^2\}$, is a basis for  $\K[x,y]/I$, precisely the one highlighting the structure of the Hilbert function.





The following results, which we prove explicitly, show that despite this ambiguity, it is always possible to choose a suitable basis compatible with the local structure:

\begin{lemma}\label{base:not:change}
Let $A=R/I$ be a local Artinian ring supported at $\zeta$ with a complete staircase basis $B$. After the affine change of coordinates $\varphi$ of \eqref{affine:tranlsation}, 
 the set $B$ remains a basis of $\varphi(A)$.
\end{lemma}
\begin{proof}
    Let $x^{\alpha}\in B$. Then 
    \[\varphi(x^\alpha)=(x_1-\zeta_1)^{\alpha_1}\cdots (x_n-\zeta_n)^{\alpha_n}=\prod_{i=1}^n \left( \sum_{k_i=0}^{\alpha_i} (-1)^{\alpha_i - k_i}\binom{\alpha_i}{k_i}x_i^{k_i}\zeta_i^{\alpha_i-k_i}\right)=\]
    \begin{equation}\label{summand:expansion}
      =  \sum_{k_1,\ldots,k_n=0}^{\alpha_1,...,\alpha_n} \left( \prod_{i=1}^n (-1)^{\alpha_i -k_i}\binom{\alpha_i}{k_i}\zeta_i^{\alpha_i-k_i}x_i^{k_i} \right).
    \end{equation}
    Remark that the linear combination \eqref{summand:expansion} above is made by $x^{\alpha}$ and lower degree terms, and all other monomials divide $x^{\alpha}$. 

Hence, since $B$ is a complete staircase, we have that \eqref{summand:expansion} is a linear combination of elements in $B$. Thus, the span of $\varphi(B)$ is the span of $B$, and since $\varphi$ is an isomorphism of rings and of vector spaces, we have that $B$ is also a basis of $\varphi(B)$. \end{proof}
This allows to assume the algebra to be supported at 0.




%


\paragraph{\textbf{Restriction on the number of bases for the local case}}
\begin{lemma}\label{lemma: BasisFromHF}
    Let $A= \K[x_1 , \ldots , x_n]/I$ be an Artinian local ring of codimension $n$. Then there exists a complete staircase of monomials $B$ whose reduction modulo $I$ is a $\K$-basis, and the number of elements in $B$ of degree $i$ is $\operatorname{HF}_A(i)$.
\end{lemma}

\begin{proof}
    By \Cref{base:not:change} and the invariance of the Hilbert function of a local Artinian algebra under affine change of coordinates, we may assume that $A$ is supported at the origin.  
    We will prove the existence of a basis $B$ of $A$ consisting of monomials forming a complete staircase, with the degree distribution prescribed by the Hilbert function. The proof is in three steps:

    \begin{itemize}
        \item Existence of a monomial basis respecting the degree distribution.  
        Let $\mathfrak{m}$ be the maximal ideal of $A$. For each $i \ge 0$, the quotient $\mathfrak{m}^i / \mathfrak{m}^{i+1}$ is spanned by the residue classes of monomials of total degree $i$.  
        Choosing a $\K$-basis of $\mathfrak{m}^i / \mathfrak{m}^{i+1}$ made of monomials for each $i$, and collecting them for all $i$, we obtain a monomial basis $B$ of $A$. The number of basis elements of degree $i$ is then $\dim_{\K} \mathfrak{m}^i / \mathfrak{m}^{i+1} = \operatorname{HF}_A(i)$.

        \item Staircase property.  
        Suppose $b \in B$ has degree $i$ and $x_j \in \{x_1,\ldots,x_n\}$. If $b x_j$ is nonzero in $\mathfrak{m}^{i+1}/\mathfrak{m}^{i+2}$, then $b \notin \mathfrak{m}^{i+1}$ (otherwise $b x_j \in \mathfrak{m}^{i+2}$, contradicting $b x_j \neq 0$).  
        This shows that whenever $b \in B$ and $b x_j \neq 0$ in $A$, the product $b x_j$ lies in the basis of degree $i+1$. Hence $B$ can be chosen to be closed under multiplication by variables, i.e. a complete staircase.

        \item Inclusion of $1$ and the variables.
        Since $\operatorname{HF}_A(0) = 1$ and $\operatorname{HF}_A(1) = n$, the constant $1$ and the classes of $x_1,\ldots,x_n$ must appear in $B$.
    \end{itemize}
\end{proof}

In our algorithmic setting, \Cref{lemma: BasisFromHF} is key to selecting admissible bases when searching for local apolar schemes. This refines \Cref{alg:basis} of the global algorithm in \cite{Alessandra} by  shrinking the search space: instead of testing all possible bases of length $r = \dim_\K A$, we can restrict to staircase bases consistent with the possible Hilbert functions.

We compare the growth of the search space in two settings:
\begin{itemize}
    \item 
\emph{Unconstrained search.}
Let $r = \dim_\K A$ and suppose we consider all monomial subsets of length $r$ from the set of all monomials of degree $\le d$. The number of such monomials is 
$N_d = \binom{n+d}{d}
$.

For fixed $d$ ($n$ resp.), $N_d$ grows polynomially in $n$ ($d$ resp.); if both grow, the growth is superpolynomial.  
The number of possible bases is therefore $
\binom{N_d}{r} = \binom{\binom{n+d}{d}}{r}$,
which becomes rapidly intractable as $n$ and $d$ increase.

\item \emph{Search constrained by the Hilbert function.}
If $\operatorname{HF}_A = (h_0,\ldots,h_d)$, then any staircase basis $B$ has exactly $h_k$ elements of degree $k$.  
The number of sets of monomials that merely satisfy the degree distribution is bounded above by $
\prod_{k=0}^d \binom{\binom{n+k-1}{k}}{h_k}$, 
where $\binom{n+k-1}{k}$ is the number of monomials of exact degree $k$.  
However, \Cref{lemma: BasisFromHF} ensures that $B$ must also be a complete staircase, which reduces again the number of admissible choices: for a given $\operatorname{HF}_A$, often only very few staircase bases exist.  

\end{itemize}
Thus, instead of choosing $r$ monomials arbitrarily from a large and rapidly growing set, we are choosing a highly structured set dictated by the Hilbert function, yielding a far smaller and more tractable search space.

\begin{example}

Consider $n=2$ variables with Hilbert function 
$\operatorname{HF}_A = (1,2,2,1,1)$,

corresponding, for instance, to $A \cong \K[x,y]/\operatorname{Ann}(x^{(4)}+y^{(3)})$, which has $\dim_\K A = r = 7$.  
Here $d = 4$ is the socle degree.

\emph{Unconstrained search.}
There are $N_4 = \binom{2+4}{4} = 15$ monomials of degree at most 4 in 2 variables.  
An unconstrained search would choose any $r=7$ monomials from these $15$, giving
$\binom{15}{7} = 6435$ possible candidate bases.

\sloppy\emph{Search constrained by $\operatorname{HF}_A$.}
The Hilbert function prescribes the degree distribution, i.e.:  $h_0=1$ element of degree 0, $h_1=2$ elements of degree 1,  
 $h_2=2$ elements of degree 2, 
 $h_3=1$ element of degree 3 and $h_4=1$ element of degree 4  
An upper bound for the number of bases respecting this distribution is
$
\binom{1}{1} \cdot \binom{2}{2} \cdot \binom{3}{2} \cdot \binom{4}{1} \cdot \binom{5}{1} 
= 1 \cdot 1 \cdot 3 \cdot 4\cdot 5 = 60$.
So at most $60$ sets of monomials satisfy the degree distribution.

\emph{Search restricted to staircase bases.}
\Cref{lemma: BasisFromHF} guarantees that the basis can be chosen as a \emph{complete staircase}, i.e.\ closed under divisibility.  
For $\operatorname{HF}_A=(1,2,2,1,1)$ there are only four such staircases: $
B_1 = \{1, x, y, x^2, xy, x^3, x^4\}$, $ 
B_2 = \{1, x, y, x^2, y^2, x^3, x^4\}$, $ 
B_3 = \{1, x, y, y^2, xy, y^3, y^4\}$ and $ 
B_4 = \{1, x, y, y^2, xy, y^3, y^4 \}$.
Thus, from $6435$ possible bases in the unconstrained setting, the search space collapses to just $4$ admissible candidates.

\end{example}
\begin{remark}
For $n=2$, staircase bases of an Artinian monomial algebra correspond to \emph{Young diagrams} determined by the Hilbert function: the monomials in the basis can be represented as lattice points in $\mathbb{N}^2$ below a monotone path.  

For $n=3$, staircase bases correspond to \emph{3-dimensional Ferrers diagrams}, 
also known as \emph{plane partitions}: the exponent set of the basis can be visualized as a stack of unit cubes in $\mathbb{N}^3$ with nonincreasing heights along each axis.  
A fixed Hilbert function imposes strong constraints on the shape of such a diagram. 

While for $n=2$ the number of staircases can be enumerated combinatorically via lattice paths, for $n=3$ this becomes a much harder problem: it amounts to counting plane partitions with a prescribed profile, for which no general closed formula is known. 
Only special families of Hilbert functions can be handled explicitly.
\end{remark}
\paragraph{\textbf{Restriction on the parameters for the local case}}

\begin{remark}\label{remark:local}    
By exploiting the local condition, we can make another  restriction. The operators $M_{x_i}$ (multiplication by variables) have a single eigenvalue at a local point  $\zeta=(1, \frac{1}{r}tr(M_{x_1}),\ldots,\frac{1}{r}tr(M_{x_n}))$, so one can enforce the algebraic constraints on the characteristic polynomials of these operators:
$$\operatorname{Char}_{M_{x_i}}(t)=\left(t-\frac{1}{r}tr(M_{x_i})\right)^r.$$
\end{remark}
These two restrictions give effective algebraic constraints to the algorithm for the local setting and we can refine it.

\begin{example}\label{Example:LocalCactus}
Consider the following polynomial:    
{\small{
\[
\begin{aligned}
F =\;& 6x_0^{(3)}x_1x_2 + 12x_0^{(2)}x_1^{(2)}x_2 + 18x_0x_1^{(3)}x_2 + 24x_1^{(4)}x_2 \\
&+ 12x_0^{(3)}x_2^{(2)} + 12x_0^{(2)}x_1x_2^{(2)} + 12x_0x_1^{(2)}x_2^{(2)} + 12x_1^{(3)}x_2^{(2)} \\
&+ 6x_0^{(2)}x_1x_2x_3 + 12x_0x_1^{(2)}x_2x_3 + 18x_1^{(3)}x_2x_3 \\
&+ 12x_0^{(2)}x_2^{(2)}x_3 + 12x_0x_1x_2^{(2)}x_3 + 12x_1^{(2)}x_2^{(2)}x_3 \\
&+ 6x_0x_1x_2x_3^{(2)} + 12x_1^{(2)}x_2x_3^{(2)} + 12x_0x_2^{(2)}x_3^{(2)} + 12x_1x_2^{(2)}x_3^{(2)} \\
&+ 12x_0^{(2)}x_3^{(3)} + 12x_0x_1x_3^{(3)} + 12x_1^{(2)}x_3^{(3)} + 6x_1x_2x_3^{(3)} + 12x_2^{(2)}x_3^{(3)} \\
&+ 48x_0x_3^{(4)} + 48x_1x_3^{(4)} + 120x_3^{(5)}
\end{aligned}
\]
}}

After building its Hankel matrix, start the iteration on the rank. The lowest possible rank is 4, since it is the number of essential variables of the polynomial. The only possible complete staircase basis is $\{1,x_1,x_2,x_3\}$, which has the following minor in the Hankel:
\[
\begin{bmatrix}
0 & 0 & 0 & 0 \\
0 & 0 & 1 & 0 \\
0 & 1 & 2 & 0 \\
0 & 0 & 0 & 0 \\
\end{bmatrix}.
\]
Since it has zero determinant, the local cactus rank cannot be 4. One checks that all the 6 possible complete staircase monomial basis of 5 elements also have minors in the Hankel with zero determinant. Thus, the rank is not 5. For $r=6$, there are 18 possible basis. We can use the Hilbert function criterion in the following way. For codimension 3, the first possible Hilbert function of a local Artinian Gorenstein algebra of length 6 is $(1,3,1,1)$. So for this $h$-vector there are 3 candidate basis $B$:
\[\{1,x_1,x_2,x_3,x_1^2,x_1^3\}, \quad \{1,x_1,x_2,x_3,x_2^2,x_2^3\}, \quad \{1,x_1,x_2,x_3,x_3^2,x_3^3\}.\]
The first two have vanishing minor in the Hankel, but the last has the associated matrix (up to scaling)
\[H_{\Lambda}^B=
\begin{bmatrix}
0 & 0 & 0 & 0 & 0 & 2 \\
0 & 0 & 1 & 0 & 0 & 2 \\
0 & 1 & 2 & 0 & 0 & 0 \\
0 & 0 & 0 & 0 & 2 & 8 \\
0 & 0 & 0 & 2 & 8 & 20 \\
2 & 2 & 0 & 8 & 20 & h_{(3,55)} \\
\end{bmatrix}
\]
which has determinant non-zero for every value of the parameter. We now build the multiplication operators with respect to $B$ and impose commutation and locality. We have a zero dimensional polynomial system with a unique solution. Thus, the rank is 6.  
\end{example}

These considerations set the stage for subsequent explicit constructions and algorithmic procedures, such as recovering local Gorenstein algebra decompositions (GAD) and Hilbert functions from commuting multiplication operators, as detailed in the following subsections.



\subsubsection{Hilbert function via multiplication operators}
In this section, we compute the Hilbert function and its symmetric decomposition from an explicit presentation of the apolar algebra $A=R/I$ via multiplication matrices in a complete staircase basis. This method applies when the local algebra is already determined (for instance, as the output of an extension algorithm for $F$), and leverages the full linear structure encoded in the multiplication operators $M_{x_i}$ to recover both the function and its decomposition.

As already observed in \Cref{remark:local}, when $A=R/I$ is a local Artinian algebra the support of the algebra $A$ is explicitly determined from the multiplication operators:
$\zeta=\left(\frac{1}{r}tr(M_{x_1}),\ldots,\frac{1}{r}tr(M_{x_n})\right).$
The multiplication matrices contain all the multiplicative structure of the algebra. For example, one can also  obtain the coordinates in the basis $B$ of the residue class of any $p\in R$.
\begin{definition}
Given an Artinian algebra $A=R/I$ and a basis $B$ of $A$, the \emph{normal form} of a polynomial $p\in R$ with respect to $I$ and $B$ is the representation of the residue class of $p$ in $A$ as a linear combination of elements in $B$.
\end{definition}


\begin{lemma}[{\cite[Proposition 3.2]{NormalForm}}]\label{lemma: NormalForm}
    Let $A=R/I$ be an Artinian $\K$-algebra and $B$ a $\K$-basis of $A$. Let $M_{x_i}$ be the matrix corresponding to the multiplication by $x_i$ in $A$ in the basis $B$. Then for every $p\in R$, the normal for of $p$ with respect to $I$ and $B$ is given by 
    $$p(M_x)(1)\in \langle B\rangle_\K,$$ where $p(M_x)$ is the linear operator in $A$ obtained by substituting each variable $x_i$ in $p$ by $M_{x_i}$.
    \end{lemma}

Note that, since by assumption $1\in B$, the coordinates of $p(M_x)(1)$ is just the first column of $p(M_x)$.

This characterization immediately provides a method to generate powers of the maximal ideal $\m_\zeta$. Indeed, $\m_\zeta^i$ is generated as an ideal by the normal forms of $(x-\zeta)^\alpha$ with $|\alpha|=i$, computed explicitly as $(M_x-\zeta I)^\alpha(1)$. We abbreviate this as $x_{\zeta}^\alpha$. Consequently, the socle degree of $A$ is the smallest integer $d$ such that $x_{\zeta}^\alpha=0$ for all $|\alpha|=d$.
Therefore, as a vector subspace, $\m^{i}$ is generated by the class of $\{x_{\zeta}^{\alpha}\}_{i\leq |\alpha|\leq d}$. Thus, we have

\[H_A(i)=\dim_{\K} \frac{\m^{i}}{\m^{i+1}}=\dim_{\K} \langle x_{\zeta}^{\alpha}\rangle_{i\leq |\alpha|\leq d} - \dim_{\K} \langle x_{\zeta}^{\alpha}\rangle_{i+1\leq |\alpha|\leq d}\]
From these observations, we obtain an effective algorithm to compute the Hilbert function $H_A$:
\vspace{0.5cm}
\begin{mdframed}[]
\begin{alg}[Hilbert Function from commuting matrices]\label{alg:HF}\end{alg}
\vspace{0.1cm}
\noindent\textbf{Input:} A list $M$ of $n$ pairwise commuting matrices  of size $r$ with a unique eigenvalue. \\
    \textbf{Output:} The Hilbert function of the Artinian and local algebra $A$, with $M_i$ the matrix corresponding to the multiplication by $x_i$ in $A$.

 \begin{enumerate}
     \item Compute the support of the algebra: $$\zeta=\frac{1}{r}(tr(M_1),\ldots,tr(M_n)).$$
     \item Define $M'_i = M_i - \zeta_i\cdot I$, $i=1,\ldots, n$.
     \item Compute the generators of the powers of the maximal ideal: set $k=0$ and $\text{GensMax}_0= \{1\}$. While $\text{GensMax}_k \neq 0$ do:


     \begin{itemize}
         \item Compute the generators of the $k+1$-th power of the maximal ideal $$\text{GensMax}_{k+1}= \{M^L \text{for $L$ in compositions$(k+1,n)$}\},$$ where \text{compositions$(k,n)$} is the set of partitions of the integer $k$ with $n$ summands. 
         \item Set $k \; \rightarrow k+1$
     \end{itemize}

     \item Define $d:=k$, the socle degree of the algebra.
     \item For $i$ from 1 to $d$ do:
     \begin{itemize}
         \item Define $\text{Max}_i$ as the matrix of vectors in $\text{GensMax}_j$ for $j$ from $i$ to $d$. 
     \end{itemize}
     \item Print Hilbert function: $H_A(i)=\text{rank}(\text{Max}_i)-\text{rank}(\text{Max}_{i+1})$ for $i$ from $1$ to $d$.
     \end{enumerate}
\end{mdframed}
\vspace{0.5cm}

\begin{example}\label{Example: HF}
    Consider the algebra
    $$A=\K[x,y,z]/\operatorname{Ann}\left(2z^{(3)}+xy+2y^{(2)}\right).$$
    It is Artinian and of length 6, with $\{1,x,y,z,z^2,z^3\}$ as a possible basis. In this basis, the matrices corresponding to the multiplication by the variables are:
{\scriptsize

\[
\textstyle
M_x =
\begin{bmatrix}
0 & 0 & 0 & 0 & 0 & 0 \\
1 & 0 & 0 & 0 & 0 & 0 \\
0 & 0 & 0 & 0 & 0 & 0 \\
0 & 0 & 0 & 0 & 0 & 0 \\
0 & 0 & 0 & 0 & 0 & 0 \\
0 & 0 & \tfrac{1}{2} & 0 & 0 & 0 \\
\end{bmatrix}
\quad
M_y =
\begin{bmatrix}
0 & 0 & 0 & 0 & 0 & 0 \\
0 & 0 & 0 & 0 & 0 & 0 \\
1 & 0 & 0 & 0 & 0 & 0 \\
0 & 0 & 0 & 0 & 0 & 0 \\
0 & 0 & 0 & 0 & 0 & 0 \\
0 & \tfrac{1}{2} & 1 & 0 & 0 & 0 \\
\end{bmatrix}
\quad
M_z =
\begin{bmatrix}
0 & 0 & 0 & 0 & 0 & 0 \\
0 & 0 & 0 & 0 & 0 & 0 \\
0 & 0 & 0 & 0 & 0 & 0 \\
1 & 0 & 0 & 0 & 0 & 0 \\
0 & 0 & 0 & 1 & 0 & 0 \\
0 & 0 & 0 & 0 & 1 & 1 \\
\end{bmatrix}.
\]
}

We clearly see that the three matrices have the unique eigenvalue 0 (as expected since it is the quotient of the annihilator of a polynomial). Then, for instance, the ideal $\m^3$ is generated as a $\K$-vector space by the column vector
{\small{
\[M_z^3(1) =
\begin{bmatrix}
0 \\
0 \\
0 \\
0 \\
0 \\
1 \\
\end{bmatrix}.\]
}}
In our basis, it corresponds to $z^3$.
All the other combinations $M_x^{\alpha_1}\cdot M_y^{\alpha_2}\cdot M_z^{\alpha_3}$ with $\alpha_1+\alpha_2+\alpha_3=3$ give the zero matrix. 
Further computations show that there are no non-zero matrices corresponding to elements in $\m^4$. Hence, the socle degree of $A$ is $3$. Similarly, $\m^2$ is generated as a $\K$-vector space by the column vectors corresponding to $\{z^2,z^3\}$:
{\small{
\[
\begin{bmatrix}
0 & 0\\
0 & 0\\
0 & 0\\
0 & 0\\
0 & 1\\
1 & 0\\
\end{bmatrix}.\]}}

Computing the different powers of the matrices and multiplying them accordingly, we obtain the Hilbert function $(1,3,1,1)$. 
\end{example}

The symmetric decomposition of the Hilbert function is similarly derived by explicitly computing annihilators $(0:\m^{i})$. Specifically, for each $i$, one solves the linear system:

\begin{equation}\label{eq: annOfMax}
\sum_{i}\lambda_i b_i x_\zeta^{\alpha}=0, \quad i\leq |\alpha|\leq d,
\end{equation}
where $B={b_j}$ is the chosen monomial basis. This system's dimension equals $\dim_\K(0:\m^{i})$, leading to the full symmetric decomposition.


\begin{remark}
Each monomial $x_\zeta^{\alpha}$ can be expressed as a linear combination of the basis elements $b_i$, and products are computed via 
$
b_i b_j = (b_i b_j)(M_x)(1)
$.
The basis $B$ itself need not be known explicitly: it can be reconstructed (up to equivalence modulo $I$) from the multiplication matrices.  
Specifically, if for some multi-index $\alpha \in \mathbb{N}^n$ the matrix $M_x^{\alpha}$ has a single nonzero entry in its first column, then $x^{\alpha}$ corresponds to a basis element.  
See \Cref{Example: SymmetricDecompFromMatrices} for an explicit instance.
\end{remark}

Using the definitions of the symmetric filtrations, we can recover the full symmetric decomposition. 

\begin{example}\label{Example: SymmetricDecompFromMatrices}
    Let us compute the symmetric decomposition of the Hilbert function obtained in \Cref{Example: HF}.
\bigskip

    Let us first see how to recover the elements of the basis from the matrices $\{M_x,M_y,M_z\}$. Since the size of the matrices is $6$ and $\{1,x,y,z\}$ is always part of a complete monomial staircase basis, we look for two other elements of the basis. We see that 
    {\scriptsize
\[
M_z^2 =
\begin{bmatrix}
0 & 0 & 0 & 0 & 0 & 0 \\
0 & 0 & 0 & 0 & 0 & 0 \\
0 & 0 & 0 & 0 & 0 & 0 \\
0 & 0 & 0 & 0 & 0 & 0 \\
1 & 0 & 0 & 0 & 0 & 0 \\
0 & 0 & 0 & 1 & 0 & 0 \\
\end{bmatrix}
\quad
M_z^3 =
\begin{bmatrix}
0 & 0 & 0 & 0 & 0 & 0 \\
0 & 0 & 0 & 0 & 0 & 0 \\
0 & 0 & 0 & 0 & 0 & 0 \\
0 & 0 & 0 & 0 & 0 & 0 \\
0 & 0 & 0 & 0 & 0 & 0 \\
1 & 0 & 0 & 0 & 0 & 0 \\
\end{bmatrix}.
\]
}
Looking at the first column of the first matrix we see that $z^2$ is in the basis, while the second matrix shows $z^3$ is in the basis. One can check that we also have $M_y^2=M_z^3$, which is consistent with the fact that $y^2$ is equal to $z^3$ modulo $I$. 

We have seen the computation of the powers of the maximal ideal. For the annihilators of the powers of the maximal ideal, applying \eqref{eq: annOfMax} we have that for instance $(0: \m^3)$ is the space of polynomials in $A$ whose coordinates in the basis $\{1,x,y,z,z^2,z^3\}$ are in the kernel of this matrix: 
\[
\begin{bmatrix}
1 & 0 & 0 & 0 & 0 & 0 \\
0 & 0 & 0 & 1 & 0 & 0 
\end{bmatrix}\]
That is, if a polynomial in $A$ is $\lambda_1+\lambda_2x+\ldots + \lambda6z_3$, it must verify $\lambda_1=\lambda_3=0$. Note that this is consistent with the fact that $\m^2$ is generated by $z^2,z^3$.

We can compute the rest of the annihilators and then the intersections with powers of maximal ideals, obtaining the decomposition 

\begin{table}[h]
\centering
\begin{tabular}{|l|l|l|l|l|}
\hline
Degree         & 0 & 1 & 2 & 3  \\ \hline
$\Delta_{Q_0}$ & 1 & 1 & 1 & 1  \\ \hline
$\Delta_{Q_1}$ & 0 & 2 & 0 & 0  \\ \hline
Total          & 1 & 3 & 1 & 1  \\ \hline
\end{tabular} \,.
\end{table}

\end{example}
\subsubsection{Recovery of Local Generalized Additive Decomposition}\label{Section:GAD}

In this subsection we describe how to recover a local Generalized Additive Decomposition (GAD) of  given $F$ (or of one of its extensions) once the support and socle degree of the desired local GAD are fixed.
This viewpoint provides a constructive procedure, starting from the prescribed data.

After a suitable change of coordinates, a homogeneous polynomial $F$ admitting a local GAD supported at a point can be expressed in the form
\begin{equation}\label{GADLocal}
    F = L^{(k)} \cdot G, \quad L \in S_1^*, \; G \in S_{d-k}^*,
\end{equation}
where $L$ is a linear form vanishing at the support point.  
Such an expression naturally arises from dehomogenizing $F$ with respect to $L$; the resulting local apolar schemes were explicitly analyzed in \cite{GAD}, where the scheme is said to be \emph{evinced by the GAD}.


The relationship between $F$ and its minimal local apolar schemes is clarified in \cite[Proposition~4]{BJPR}: for schemes supported at $[1:0:\dots:0]$, the dehomogenization $f = F(Y_0=1)$ and the form $g$ defining the apolar scheme $\Gamma = \operatorname{Spec}(R/\operatorname{Ann}(g))$  share the same degree-$d$ tail. By an affine change of coordinates, this applies to any support point.


Equivalently, a minimal local apolar scheme supported at $[1:0:\dots:0]$ is \emph{evinced} by a form 
\[
G = Y_0^{(d'-k)} \cdot H \in S_{d'}^*, \quad d' \ge d,
\]
such that $G(Y_0=1)$ has the same degree-$d$ tail as $f=F(Y_0=1)$.

\begin{definition}\label{def:extension}
Let $F \in S_d^*$ be a homogeneous polynomial.  
An \emph{extension} of $F$ is a form $G \in S_{d'}^*$ with $d' \ge d$ such that there exists a linear form $L \in S_1^*$ and an integer $m = d' - d \ge 0$ for which 
\[
\partial_L^{\,m} G = F.
\]
When $m=0$ (equivalently $d'=d$), $G$ coincides with $F$; for $m>0$, $G$ is a genuine higher-degree extension of $F$.
\end{definition}
Given $F\in S_d^*$, the local cactus algorithm finds a functional $\Lambda\in R^*$ that extends $F(Y_0=1)\in R_{\leq d}$, so that the algebra $A=R/\ker H_\Lambda$ is local and has minimal length. The minimal apolar scheme is then $Z=\operatorname{Spec}(A)$. The support $\zeta\in \K^n$ is recovered from the multiplication matrices, while the remaining structure is obtained from the dual generator $G$ associated to $\Lambda$.

\begin{remark}
By the Generalized Macaulay Correspondence \Cref{Thm: GeneralisedMacaulayCorrespondence}, a local Gorenstein algebra $A = R/\ker H_\Lambda$ of socle degree $d'$ supported at $\zeta$ corresponds to a functional 
\[
\Lambda = \mathds{1}_\zeta \circ G(\partial),
\]
where $G \in S_{d'}^*$ is a dual generator (possibly an extension of $F$).

As shown in \cite[Proposition~4.3]{GAD}, $Z$ is evinced by a GAD of $F$ with $G=F$ and $d'=d$ if and only if
\[
\mathfrak{p}_L^{d+1} \subseteq I(Z),
\]
where $\mathfrak{p}_L$ is the homogeneous maximal ideal of the support point $[L] \in \mathbb{P}^n$. For a local scheme supported at $[L]$, the socle degree of $R/\ker H_\Lambda$ is the smallest $m$ such that $\mathfrak{p}_L^{m+1} \subseteq I(Z)$. Hence:
\begin{itemize}
    \item If $\operatorname{socdeg}(R/\ker H_\Lambda) \le d$, then $G=F$ (no genuine extension is needed).  
    \item If $\operatorname{socdeg}(R/\ker H_\Lambda) > d$, then $G$ has degree $d'>d$ and is a genuine extension of $F$.
\end{itemize}
\end{remark}

The approach to recover a GAD of either $F$ or of an extension of $F$  depends on whether the socle degree does not exceed, or exceeds, the degree of $F$.

\medskip

\paragraph{\textbf{Socle degree $\leq  d$:}}  In this case, a GAD of $F$ exists explicitly as:
\[F=(x_0+\zeta_1x_1+\cdots+\zeta_nx_n)^{(d-k)}H\]
where $H$ is recovered by performing a suitable affine coordinate shift to move $\zeta$ to the origin. After the shift, one can easily factor out the maximal possible power of $x_0$ from $F$.

%
%
So assume we have the minimal apolar scheme $Z$ to $F\in S_d^*$, and we know the support $\zeta=(\zeta_1,\ldots,\zeta_n)\neq 0$. Therefore $Z$ is defined locally by $\text{Spec}(R/\ker(H_\Lambda))$ for $\Lambda$ an infinite series. Consider the change of coordinates $\varphi$ in $S$ induced by 
\[x_0\mapsto x_0 \quad \quad x_i=x_i-\zeta_ix_0 \quad \forall i=1,\ldots, n\]
The induced affine change of coordinates $\varphi$ on $R$ is given by $x_i\mapsto x_i-\zeta_i$. Now $Z'=\varphi(Z)$ is defined locally by $\text{Spec}\left(R/\ker H_{(\varphi')^*(\Lambda)}\right)$. By \Cref{lemma:DualChangeBasis}, $\Lambda$ is now a polynomial $\Tilde{f}\in \K_{dp}[Y_1,\ldots, Y_n]$, and by hypothesis $f$ has degree $k\leq d$. Therefore we have that $Z'$ is the minimal local apolar scheme to
\[\varphi^*(F)= \operatorname{Homog}(\Tilde{f},x_0, d).\]
Dealing carefully with divided powers we obtain a decomposition 
\[\varphi^*(F)=Y_0^{d-k}\cdot \Tilde{H}.\]
Hence,
\[F=(\varphi^{-1})^*(Y_0^{(d-k)})\cdot (\varphi^{-1})^*(\Tilde{H})=
(x_0+\zeta_1x_1+\ldots+\zeta_nx_n)^{(d-k)}\cdot (\varphi^{-1})^*(\Tilde{H}),\]
which gives the desired decomposition of $F$.

\medskip

In terms of an algorithm, once we have the matrices corresponding to the multiplication by the variables in the algebra $R/\ker H_\Lambda$, for some $\Lambda\in R^*$, we can easily compute the socle degree and the support of the scheme $\zeta$. If the socle degree is not greater than the degree of $F\in S_d^*$, then we can apply the method described above, illustrated in the following example.

\begin{example}
    Consider the following homogeneous polynomial in divided powers: 

\[
\begin{aligned}
 F=&
3x_0^{(3)} x_1 + 8x_0^{(2)} x_1^{(2)} + 12x_0 x_1^{(3)} - 2x_0^{(2)} x_1 x_2 - 4x_0 x_1^{(2)} x_2 + x_0^{(2)} x_2^{(2)} +  \\
&3x_0 x_1 x_2^{(2)} + 4x_1^{(2)} x_2^{(2)} - 3x_0 x_2^{(3)} - 6x_1 x_2^{(3)} + 6x_2^{(4)}.
\end{aligned}
\]
Assume we know that its minimal local apolar scheme is supported at $[1:2:-1]$, and it is evinced by a decomposition of $F$. Following the discussion above, we move the support to the affine 0. By \Cref{prop: ChangeCoordInDivPowers}, first we convert it to a standard polynomial, $\Psi(F)$, and then we make the change of variables $\Tilde{\varphi}$, $x_0 \mapsto x_0-2x_1+x_2$, so that it reduces to a quartic in which $x_0^2$ factors out:
\[
(\Tilde{\varphi} \circ \Psi )(F)=x_0^2 \left(\frac{1}{2}x_0x_1-x_1^2+\frac{1}{2}x_1x_2+\frac{1}{4}x_2^2\right).
\]
Substituting back $x_0 \mapsto x_0+2x_1-x_2$ gives
\[
\Psi(F)=(x_0+2x_1-x_2)^2\left(\frac{1}{2}x_0x_1 + \frac{1}{4}x_2\right).
\]
Converting back to divided powers yields the GAD decomposition of $F$.
\end{example}

\begin{remark}
    In summary, when the socle degree of the minimal local apolar scheme $Z$ satisfies $\operatorname{socdeg}(Z) \le d$, the recovery of the GAD proceeds is extremely simple:
    \begin{itemize}
        \item Translate the support $\zeta$ to the origin by an affine change of coordinates, \item Factor out the maximal possible power of $x_0$ from $F$, 
        \item Revert the change of coordinates.
    \end{itemize}The resulting expression gives the desired GAD of $F$.
\end{remark}

\paragraph{\textbf{Socle degree $> d$:}} 
In this case, the minimal local apolar scheme $Z$ is evinced by a polynomial of degree higher than $F$, referred to as an \emph{extension of $F$}.  Recall from \Cref{def:extension} that in this case the minimal local scheme is evinced by a genuine extension $G$ of $F$, with $d'>d$ and $\partial_L^{d'-d} G = F$.

After an affine change of coordinates sending $\zeta$ to the origin, the functional becomes $
\varphi^*(\Lambda) = \mathds{1}_0 \circ g(\partial),
$ 
where $g \in K_{dp}[Y_1,\dots,Y_n]$ represents the dehomogenization of $G$ at $Y_0=1$.

The key observation from \cite[Proposition~4]{BJPR} is that $g$ agrees with $F(Y_0=1)$ \emph{up to degree $d$}: the coefficients in degrees $\le d$ are completely determined by $F$.  
However, for degrees $> d$, then $g$ must contain additional terms that encode the full structure of $R/\ker H_\Lambda$.  
These higher-degree coefficients are determined uniquely by the values of $\Lambda$ on monomials of degree $>d$, which can be computed as
\[
\varphi^*(\Lambda(x^\beta)) = \Lambda(\operatorname{NF}_B((x+\zeta)^\beta)),
\]
where $\operatorname{NF}_B((x+\zeta)^\beta)$ denotes the normal form of $(x+\zeta)^\beta$ in the local algebra $R/\ker H_\Lambda$ with respect to a basis $B$.  

Once all coefficients of $g$ are determined (degrees $\le d$ from $F$, degrees $>d$ from $\Lambda$), the extension $G$ is obtained by homogenizing $g$ back to degree $d'=\operatorname{socdeg}(Z)$.  
Finally, reverting the coordinate change $\varphi$ expresses $G$ in the original coordinates, yielding a GAD that evinces $Z$.

\begin{example}
    We can test this method with \cite[Example 4.6]{GAD}, with a previous change of coordinates.
{\small{
\[
\begin{aligned}
F= &144x_0^{(3)} + 284x_0^{(2)}x_1 + 444x_0x_1^{(2)} + 624x_1^{(3)} + 150x_0^{(2)}x_2 \\
&+ 150x_0x_1x_2 + 150x_1^{(2)}x_2 + 140x_0x_2^{(2)} + 140x_1x_2^{(2)} + 360x_2^{(3)} \\
&+ 428x_0^{(2)}x_3 + 708x_0x_1x_3 + 1028x_1^{(2)}x_3 + 660x_0x_2x_3 + 660x_1x_2x_3 \\
&+ 400x_2^{(2)}x_3 + 1256x_0x_3^{(2)} + 1816x_1x_3^{(2)} + 2040x_2x_3^{(2)} + 3552x_3^{(3)} \\
&+ 216x_0^{(2)}x_4 + 76x_0x_1x_4 + (-84)x_1^{(2)}x_4 + (-30)x_0x_2x_4 + (-30)x_1x_2x_4 \\
&+ (-140)x_2^{(2)}x_4 + 292x_0x_3x_4 + 12x_1x_3x_4 + (-420)x_2x_3x_4 + 184x_3^{(2)}x_4 \\
&+ (-576)x_0x_4^{(2)} + (-436)x_1x_4^{(2)} + (-90)x_2x_4^{(2)} + (-1012)x_3x_4^{(2)} + 936x_4^{(3)} \\
&+ 452x_0^{(2)}x_5 + 872x_0x_1x_5 + 1352x_1^{(2)}x_5 + 450x_0x_2x_5 + 450x_1x_2x_5 \\
&+ 420x_2^{(2)}x_5 + 1324x_0x_3x_5 + 2164x_1x_3x_5 + 1980x_2x_3x_5 + 3848x_3^{(2)}x_5 \\
&+ 628x_0x_4x_5 + 208x_1x_4x_5 + (-90)x_2x_4x_5 + 836x_3x_4x_5 + (-1708)x_4^{(2)}x_5 \\
&+ 1416x_0x_5^{(2)} + 2676x_1x_5^{(2)} + 1350x_2x_5^{(2)} + 4092x_3x_5^{(2)} + 1824x_4x_5^{(2)} + 4428x_5^{(3)}
\end{aligned}
\]
}}
    When we run the algorithm for the we find that for degree 6, which the number of essential variables of $F$ and thus the minimum possible cactus rank of $F$, the only possible complete staircase is $\{1,x_1,\ldots, x_5\}$, for which the matrices $\{M_{x_1}, \ldots , M_{x_5}\}$ commute and have unique eigenvalues. These unique eigenvalues are $1,0,2,-1,3$ respectively.  Hence, the minimal (local) apolar scheme is supported at the affine point $(1,0,2,-1,3)$. Via the methods described above, using the matrices $M_{x_i}$ one finds that the Hilbert of the local defining algebra $R/\ker H_\Lambda$ is $(1,2,1,1,1)$. Hence, the socle deree is $4>\deg F$. 

    \bigskip
    
    We consider every monomial $x^\alpha$ of degree 4, and compute the normal form of $(x+\zeta)^\alpha$ in the basis $\{1,x_1,\ldots,x_5\}$. It turns out that the only nonzero result is 
\[\operatorname{NF}_B((x_2-0)^4)=6\overline{x_5}-18.\]
    Looking at the dehomogenization of $F$, we see that $$\Lambda(x_2^4)=\Lambda(6x_5-18)=6\cdot 452-18\cdot 144=120.$$
    
    Thus, if $\varphi$ is the homogeneous change of coordinates sending $\zeta$ to 0, the scheme is defined locally by $\operatorname{Ann}(g)$, where 
    \[g=((\varphi^*(F))(Y_0=1)+ 120 x_2^{(4)}=
    120x_2^{(4)} + 360x_2^{(3)} + 120x_2^{(2)}x_3 + 20x_1^{(2)} + 140x_2^{(2)} + \]\[ 360x_2x_3 + 120x_3^{(2)} + 120x_2x_4 + 140x_1 + 150x_2 + 140x_3 + 360x_4 + 20x_5 + 144.
\]
    Then 
    $Z$ is evinced by $Y_0^{(0)}\cdot \operatorname{Homog}(g, Y_0, 4)$ (which agrees with \cite[Example 4.6]{GAD}), or in the original system of coordinates by 
    \[G=(Y_0+Y_1+2Y_3-Y_4 + 3Y_5)^{(0)}\cdot (\varphi^{-1})^*)(\operatorname{Homog}(g, Y_0, 4)).\]
\end{example}
In practice, the algorithm proceeds by
\begin{itemize}
    \item Translating the support $\zeta$ to the origin,
    \item Determining the higher-degree terms of the dual generator $g$ from the normal forms in $R/\ker H_\Lambda$,
    \item  Homogenizing $g$ to obtain the extension $G$, 
    \item Reverting the coordinate change to recover the GAD in the original variables.
\end{itemize}

\subsubsection{Global procedure for a minimal local apolar scheme and GAD}\label{Section:GlobalGAD}

The two cases considered above in the previous section whether the socle degree of an apolar scheme of a given $F$ is either $\le d$ or $> d$, can be unified into a single constructive procedure.  
Given a homogeneous polynomial $F \in S_d^*$, the goal is to determine a minimal \emph{local} apolar scheme $Z$ to $F$ and recover the corresponding GAD of $F$ (if $\operatorname{socdeg}(Z) \le d$) or of a genuine extension $G$ of $F$ (if $\operatorname{socdeg}(Z) > d$).

The standard cactus rank algorithm looks for a minimal apolar scheme to $F$, which, in general, may have support at multiple points.  
Here we are interested only in \emph{local} schemes, i.e., zero-dimensional schemes supported at a single point.  
As discussed in the local cactus rank algorithm, the search for a minimal apolar scheme can be constrained to the local setting by incorporating the locality conditions directly into the cactus rank procedure.  
This is achieved by imposing that all multiplication matrices $M_{x_i}$ not only commute but also have a \emph{single eigenvalue}, which coincides with the $i$-th coordinate of the support point.  

Once a minimal local scheme is found, its support $\zeta$ and socle degree $k$ are determined from the multiplication matrices and the Hilbert function.  
The recovery of the GAD then follows the two cases described earlier.

This leads to the following unified algorithm.

\vspace{0.5cm}
\begin{mdframed}[]
\begin{alg}[Minimal local apolar scheme and GAD from $F$]\label{alg:GlobalGAD}
\end{alg}
\vspace{0.1cm}
\noindent\textbf{Input:}  
Homogeneous divided power polynomial $F \in S_d^*$ of degree $d$.

\noindent\textbf{Output:}  
Minimal local apolar scheme $Z = \operatorname{Spec}(R/\ker H_\Lambda)$ to $F$ and a GAD decomposition of $F$ (if $\operatorname{socdeg}(Z) \le d$) or of an extension $G$ of $F$ (if $\operatorname{socdeg}(Z) > d$).

\begin{enumerate}
    \item \textbf{Construct the apolar algebra.}  
    Dehomogenize $F$ as $f = F(Y_0=1)$ and construct the Hankel matrix $H_\Lambda(h)$ with parameters $h_\alpha$ for $|\alpha|>d$.

    \item \textbf{Find the minimal local scheme.}

        \begin{enumerate}
            
        \item\label{alg:local:CHStar} Set $r=\text{rank } H_f^{\ostar}$ (see \Cref{notation: InitialRank}).

    \item\label{Step3} For each admissible Hilbert function $ h = (1,h_1,\dots,h_{d'}) $ of length $ r $, select $B\subseteq R$ a complete staircase of monomials with $|B|=r$ with $h_i$ monomials of degree $i$ and do: 
    \begin{enumerate}
     
        \item\label{alg:localC:h's} Find $h$'s such that:
        \begin{enumerate}
            
            \item $H_{\Lambda(h)}^B$ has nonzero determinant
            \item The multiplication operators $(M_{x_i})^t$ commute for all $i=1,\ldots,n$.
            \item The multiplication operators $M_{x_i}$ have a unique eigenvalue: 
            \[
            \text{char}_{M_{x_i}}(t) = (t - \tfrac{1}{r} \mathrm{tr}(M_{x_i}))^r \]
            
        \end{enumerate}
    
        \item If found, go to \Cref{Alg:Step:SupportAndSocle}. If not, go to \Cref{Step3} with another choice of $B$. If all choices of bases $B$ of length $r$ are performed go to \Cref{alg:localC:increase:r}.
   
    \end{enumerate}
    \item\label{alg:localC:increase:r} Set $r\to r+1$ and go to \Cref{Step3}.
        \end{enumerate}

    This yields a local scheme $Z = \operatorname{Spec}(R/\ker H_\Lambda)$ supported at a single point.

    \item\label{Alg:Step:SupportAndSocle} \textbf{Determine the support and socle degree.}  
    \begin{enumerate}
        \item 
    Compute the support
    \[
    \zeta = \left(\frac{1}{r}\mathrm{tr}(M_{x_1}), \dots, \frac{1}{r}\mathrm{tr}(M_{x_n})\right).
    \]  
\item     Translate $\zeta$ to the origin via a homogeneous coordinate change $\varphi$.  
   \item 
    From the Hilbert function $(1, h_1, \ldots, h_k)$ of $R/\ker H_\Lambda$, determine the socle degree
    \[
    k = \operatorname{socdeg}(R/\ker H_\Lambda),
    \] i.e. the largest $j$ such that $h_j \neq 0$.

    \end{enumerate}
    \item \textbf{Recover the GAD (case distinction based on $k$).}
    \begin{itemize}
        \item \(\mathbf{k \le d}\) (no extension needed):  
        \begin{enumerate}
            \item Write $\varphi^*(F)$ in the form $Y_0^{(d-k)} \cdot H$ by factoring out the maximal possible power of $Y_0$.  
            \item Apply $\varphi^{-1}$ to recover the GAD of $F$ in the original coordinates.
        \end{enumerate}
        
        \item \(\mathbf{k > d}\) (extension needed):  
        \begin{enumerate}
            \item Dehomogenize $\varphi^*(F)$: $f = \varphi^*(F)(Y_0=1)$ (coefficients for degrees $\le d$ of $g$).  
            \item For each $k>|\alpha|>d$, compute $\Lambda(\operatorname{NF}_B((x+\zeta)^\alpha))$ to determine the higher-degree coefficients of $g$.  
            \item Homogenize $g$ to degree $k$ to obtain the extension $G$.  
            \item Apply $\varphi^{-1}$ to recover the GAD of $G$ in the original coordinates.
        \end{enumerate}
    \end{itemize}

    \item \textbf{Output.}  
    Return $Z = \operatorname{Spec}(R/\ker H_\Lambda)$ and the GAD of $F$ (if $k \le d$) or of $G$ (if $k > d$).
\end{enumerate}
\end{mdframed}

\vspace{0.5cm}

\begin{example}\label{Ex:GlobalAlg_k>d}
Consider the homogeneous divided power polynomial
$
F = x_0^{(2)} + x_0x_1 + x_1^{(2)} + x_0x_2 \in S_2^*$.
We apply \Cref{alg:GlobalGAD} to determine the minimal local apolar scheme $Z$ to $F$ and the corresponding GAD.

Setting $Y_0=1$, we obtain $
f = 1 + y_1 + y_1^{(2)} + y_2 \in K_{dp}[y_1,y_2]$. The coefficients of $f$ determine $\Lambda(y^\alpha)$ for all $|\alpha|\le d$.

We construct the Hankel matrix $H_\Lambda(h)$.
Imposing the local condition, we restrict to complete staircases compatible with Hilbert functions of local Gorenstein algebras, and require that all multiplication matrices $M_{x_i}$ commute and have a unique eigenvalue.  
A suitable basis is $
B = \{1, x_1, x_2, x_1^2, x_1x_2, x_2^2, x_1^3\}$,
with Hilbert function
$
H = (1,3,2,1)$ which already shows that the socle degree is $k=3 > d=2$. Thus $Z$ is not evinced by a GAD of $F$, but by a GAD of a degree-$k$ extension $G$ of $F$.

From the eigenvalues of the multiplication matrices, we find $\zeta = (0,0,0)$. No affine translation is needed.

Since $k>d$, we compute the higher degree terms of $g=\varphi^*(\Lambda)$ by evaluating $\Lambda$ on the normal forms of monomials of degree $3$.  
For $y_1^3$, the normal form is $x_1^3\in B$, and we find $\Lambda(x_1^3)=1$.  
All other monomials of degree $3$ reduce to $0$ in $B$.  
Thus
$
g = 1 + y_1 + y_2 + y_1^{(2)} + y_1^{(3)}$.

Homogenizing $g$ to degree $k=3$ gives
$
G = x_0^{(3)} + x_0^{(2)}x_1 + x_0^{(2)}x_2 + x_0x_1^{(2)} + x_1^{(3)}.$

Since $\zeta=0$, the linear form is $L=x_0$.  
Factoring $L^{(k-d)}=x_0^{(1)}$ we obtain
$G = x_0^{(1)} \cdot \left(x_0^{(2)} + x_0x_1 + x_0x_2 + x_1^{(2)} + x_1^{(3)}\right)$.

This is a valid local GAD evincing $Z$.
\end{example}

\subsubsection{Hilbert function via Hankel operators}
In this section, we recover the Hilbert function and its symmetric decomposition directly from a polynomial $F\in R^*$.
This step is essential for determining the socle degree of the apolar algebra $A=R/\ker H_\Lambda$, which in turn governs whether a GAD of $F$ exists or whether a genuine extension is needed. 
Our method, following \cite[Section~1]{BJPR}, exploits the filtrations $\mathfrak{m}^i$ (which corresponds to partial derivatives of $F$ of order at least $i$) and $(0:\mathfrak{m}^i)$ (which corresponds to partial derivatives of degree at most $i$).  
These can be accessed directly via Hankel operators, without explicitly computing a basis or multiplication matrices for $A$. 



\begin{definition}
    \begin{itemize}
        \item \leavevmode Let $\m=(x_1,\ldots, x_n)$ and $p\in R$. The \emph{order} of $p$ is the integer $k\geq 0$ such that $p\in \m^k$ but $p\not \in \m^{k+1}$.
        \item Let $F\in \K_{dp}[Y_1,\ldots, Y_n]$ and $p\in \operatorname{Im}(H_f)$. We call the order of $p$ the maximum order of a $q\in R$ such that $p=q\aprod f$.
    \end{itemize}
     
\end{definition}

Under the isomorphism $R/\ker H_F \cong \operatorname{Im}(H_F)$, the ideal $\m^{i}$ corresponds to partial derivatives of $F$ of order at least $i$, while $(0:\m^{i})$ corresponds to partial derivatives of degree at most $i$. The symmetric decomposition $\Delta_{Q_a}(i)$ counts the partials of $F$ of degree $i$ and order $d-a-i$, modulo partials of smaller degree and higher order.

From a computational viewpoint, the space of partials of order at least $k$ is spanned by $\{H_F(x^\alpha)\}_{|\alpha|\ge k}$, i.e., by the columns of the Hankel matrix indexed by monomials of degree at least $k$. The degree of each partial is read from the corresponding row indexing.


\begin{example}
    We compute the Hilbert function of the apolar algebra $R/\operatorname{Ann}(f)$ for $f=x_1^{(3)}x_2+x_3^{(3)} + x_2^{(2)}$. Let's start by computing its Hankel matrix. Since $f\in R^*$ is 0 on $R_{>4}$, only the matrix indexed by monomials of degree at most $3$ is needed\footnote{The first column corresponds to $1\aprod f=f$, so we can also eliminate it.}.

\[
\resizebox{\textwidth}{!}{$
\begin{array}{c|ccccccccccccccccccc}
                 & x_1 & x_2 & x_3 & x_1^2 & x_1x_2 & x_1x_3 & x_2^2 & x_2x_3 & x_3^2 & x_1^3 & x_1^2x_2 & x_1^2x_3 & x_1x_2^2 & x_1x_2x_3 & x_1x_3^2 & x_2^3 & x_2^2x_3 & x_3^3 \\
\hline 
0               & 0 & 0 & 0 & 0 & 0 & 0 & 1 & 0 & 0 & 0 & 0 & 0 & 0 & 0 & 0 & 0 & 0 & 1 \\
x_1             & 0 & 0 & 0 & 0 & 0 & 0 & 0 & 0 & 0 & 0 & 1 & 0 & 0 & 0 & 0 & 0 & 0 & 0 \\
x_2             & 0 & 1 & 0 & 0 & 0 & 0 & 0 & 0 & 0 & 1 & 0 & 0 & 0 & 0 & 0 & 0 & 0 & 0 \\
x_3             & 0 & 0 & 0 & 0 & 0 & 0 & 0 & 0 & 1 & 0 & 0 & 0 & 0 & 0 & 0 & 0 & 0 & 0 \\
x_1^2           & 0 & 0 & 0 & 0 & 1 & 0 & 0 & 0 & 0 & 0 & 0 & 0 & 0 & 0 & 0 & 0 & 0 & 0 \\
x_1x_2          & 0 & 0 & 0 & 1 & 0 & 0 & 0 & 0 & 0 & 0 & 0 & 0 & 0 & 0 & 0 & 0 & 0 & 0 \\
x_1x_3          & 0 & 0 & 0 & 0 & 0 & 0 & 0 & 0 & 0 & 0 & 0 & 0 & 0 & 0 & 0 & 0 & 0 & 0 \\
x_2^2           & 0 & 0 & 0 & 0 & 0 & 0 & 0 & 0 & 0 & 0 & 0 & 0 & 0 & 0 & 0 & 0 & 0 & 0 \\
x_2x_3          & 0 & 0 & 0 & 0 & 0 & 0 & 0 & 0 & 0 & 0 & 0 & 0 & 0 & 0 & 0 & 0 & 0 & 0 \\
x_3^2           & 0 & 0 & 1 & 0 & 0 & 0 & 0 & 0 & 0 & 0 & 0 & 0 & 0 & 0 & 0 & 0 & 0 & 0 \\
x_1^3           & 0 & 1 & 0 & 0 & 0 & 0 & 0 & 0 & 0 & 0 & 0 & 0 & 0 & 0 & 0 & 0 & 0 & 0 \\
x_1^2x_2        & 1 & 0 & 0 & 0 & 0 & 0 & 0 & 0 & 0 & 0 & 0 & 0 & 0 & 0 & 0 & 0 & 0 & 0 \\
x_1^2x_3        & 0 & 0 & 0 & 0 & 0 & 0 & 0 & 0 & 0 & 0 & 0 & 0 & 0 & 0 & 0 & 0 & 0 & 0 \\
x_1x_2^2        & 0 & 0 & 0 & 0 & 0 & 0 & 0 & 0 & 0 & 0 & 0 & 0 & 0 & 0 & 0 & 0 & 0 & 0 \\
x_1x_2x_3       & 0 & 0 & 0 & 0 & 0 & 0 & 0 & 0 & 0 & 0 & 0 & 0 & 0 & 0 & 0 & 0 & 0 & 0 \\
x_1x_3^2        & 0 & 0 & 0 & 0 & 0 & 0 & 0 & 0 & 0 & 0 & 0 & 0 & 0 & 0 & 0 & 0 & 0 & 0 \\
x_2^3           & 0 & 0 & 0 & 0 & 0 & 0 & 0 & 0 & 0 & 0 & 0 & 0 & 0 & 0 & 0 & 0 & 0 & 0 \\
x_2^2x_3        & 0 & 0 & 0 & 0 & 0 & 0 & 0 & 0 & 0 & 0 & 0 & 0 & 0 & 0 & 0 & 0 & 0 & 0 \\
x_2x_3^2        & 0 & 0 & 0 & 0 & 0 & 0 & 0 & 0 & 0 & 0 & 0 & 0 & 0 & 0 & 0 & 0 & 0 & 0 \\
x_3^3           & 0 & 0 & 0 & 0 & 0 & 0 & 0 & 0 & 0 & 0 & 0 & 0 & 0 & 0 & 0 & 0 & 0 & 0 \\
\end{array}
$}
\]

    The index on the left indicates the degree of the row. 
    
    From this matrix we can easily see the Hilbert function of $R/\operatorname{Ann}(f)$: the space $(0:\m^{3})/(0:\m^2)$ is the space of partials of degree exactly $3$, in this case only the first two columns correspond to partials of degree 3, and since they are linearly independent, $H_A(3)=2$. Now, in order to compute the partials of degree exactly 2, we have to remove the two columns corresponding to degree 3 partials, and we are left with the third, fourth and fifth column, which are linearly independent. Hence, $H_A(2)=3$. Finally, we know that $H_A(1)$ is the number of essential variables of $f$, in this case 3. We reach the same conclusion by eliminating the first 5 columns and looking at the degree 1 partials of the Hankel.  

    For the symmetric decomposition, we take the degree 3 partials and we look at the columns. The two degree 3 partials correspond to columns indexed by degree 1 monomials, so their order is one. Hence, $\Delta_{Q_0}(3)=2$, $\Delta_{Q_1}(3)=\Delta_{Q_2}(3)=0$.  
    The 3 degree 2 partials have order 1,2 and 2. Hence, $\Delta_{Q_0}(2)=2, \; \Delta_{Q_1}(2)=1$. Similarly, $\Delta_{Q_0}(1)=2, \Delta_{Q_1}(1)=1$. We obtain the symmetric decomposition of \Cref{example:table:HFdecomp}.
From this discussion we extract an explicit procedure to compute $H_A$ and its symmetric decomposition.

\begin{table}[h!] 
\centering
\caption{Hilbert function decomposition of $\K[x_1,x_2,x_3]/Ann(f)$} \label{example:table:HFdecomp}
\begin{tabular}{|l|l|l|l|l|l|}
\hline
Degree         & 0 & 1 & 2 & 3 & 4 \\ \hline
$\Delta_{Q_0}$ & 1 & 2 & 2 & 2 & 1 \\ \hline
$\Delta_{Q_1}$ & 0 & 1 & 1 & 0 & 0 \\ \hline
$\Delta_{Q_2}$ & 0 & 0 & 0 & 0 & 0 \\ \hline
Total          & 1 & 3 & 3 & 2 & 1 \\ \hline
\end{tabular}
\end{table}
\end{example}

From this example and the discussion above, we can deduce an algorithm to compute the Hilbert function and its decomposition from any $f\in \K_{dp}[Y_1,\ldots, Y_n]$:

\vspace{0.5cm}
\begin{mdframed}[]
\begin{alg}[Hilbert Function decomposition]\label{alg:SymmetricHF}\end{alg}
\vspace{0.1cm}
\noindent\textbf{Input:} A polynomial $f\in \K_{dp}[Y_1,\ldots, Y_n]$. \\
    \textbf{Output:} The Hilbert function of the associated graded algebra of \\ $\K[x_1,\ldots, x_n]/\operatorname{Ann}(f)$ and its symmetric decomposition.

\begin{enumerate}
    \item Construct the matrix $H_{f}$ indexed by monomials of degree at most $\deg f-1$.
    \item Set $i=\deg f-1$.
    \item\label{alg: item: i} Let $v$ be the vector of columns of $H_F$ with at least one nonzero entry on the rows indexed by monomials of degree $i$.  
    
    \item  For $a$ from $\deg f-1$ to 1 do: 
    \begin{itemize}
        \item Take $w$ a subvector of $v$ of the columns in $H_f$ indexed by monomials of degree $d-i-a$.
        \item Set $\Delta_{Q_a}(i)=\text{rank}(w)$.
        \item Eliminate in $v$ the columns in $w$.
        \item $a\to a-1$.        
        \end{itemize}
        \item Eliminate in $H_f$ the columns in $v$.
        \item $r\to i-1$
        \item Return to step \Cref{alg: item: i}.
        
    \end{enumerate}
\end{mdframed}
\vspace{0.5cm}

\sloppy We can use a reverse engineering strategy to determine a polynomial $f\in \K_{dp}[Y_1,\ldots, Y_n]$ such that $\K[x_1,\ldots, x_n]/\operatorname{Ann}(f)$ has a given Hilbert function. 

\begin{example}
    We can test this with the previous example. We take the Hilbert function decomposition in \Cref{example:table:HFdecomp}, from which we know that the number of variables is 3 and the degree of $f$ is 4. Therefore we consider 
   {\small{ \[
\begin{aligned}
f =\;& a_0 
+ a_1 x_1 + a_2 x_2 + a_3 x_3 \\
&+ a_4 x_1^2 + a_5 x_1 x_2 + a_6 x_1 x_3 + a_7 x_2^2 + a_8 x_2 x_3 + a_9 x_3^2 \\
&+ a_{10} x_1^3 + a_{11} x_1^2 x_2 + a_{12} x_1^2 x_3 + a_{13} x_1 x_2^2 + a_{14} x_1 x_2 x_3 + a_{15} x_1 x_3^2 \\
&+ a_{16} x_2^3 + a_{17} x_2^2 x_3 + a_{18} x_2 x_3^2 + a_{19} x_3^3 \\
&+ a_{20} x_1^4 + a_{21} x_1^3 x_2 + a_{22} x_1^3 x_3 + a_{23} x_1^2 x_2^2 + a_{24} x_1^2 x_2 x_3 + a_{25} x_1^2 x_3^2 \\
&+ a_{26} x_1 x_2^3 + a_{27} x_1 x_2^2 x_3 + a_{28} x_1 x_2 x_3^2 + a_{29} x_1 x_3^3 \\
&+ a_{30} x_2^4 + a_{31} x_2^3 x_3 + a_{32} x_2^2 x_3^2 + a_{33} x_2 x_3^3 + a_{34} x_3^4
\end{aligned}
\]
    }}
 and we build its Hankel operator matrix:
\[\resizebox{\textwidth}{!}{$
\begin{array}{r|cccccccccccccccccccc}
& 1 & x_1 & x_2 & x_3 & x_1^2 & x_1x_2 & x_1x_3 & x_2^2 & x_2x_3 & x_3^2 & x_1^3 & x_1^2x_2 & x_1^2x_3 & x_1x_2^2 & x_1x_2x_3 & x_1x_3^2 & x_2^3 & x_2^2x_3 & x_3^3 \\
\hline 
 & 1 & 2 & 3 & 4 & 5 & 6 & 7 & 8 & 9 & 10 & 11 & 12 & 13 & 14 & 15 & 16 & 17 & 18 & 19 & 20 \\
\hline
1 & a_0   & a_1 & a_2 & a_3 & a_4 & a_5 & a_6 & a_7 & a_8 & a_9 & a_{10} & a_{11} & a_{12} & a_{13} & a_{14} & a_{15} & a_{16} & a_{17} & a_{18} & a_{19} \\
x_1 & a_1 & a_4 & a_5 & a_6 & a_{10} & a_{11} & a_{12} & a_{13} & a_{14} & a_{15} & a_{20} & a_{21} & a_{22} & a_{23} & a_{24} & a_{25} & a_{26} & a_{27} & a_{28} & a_{29} \\
x_2 & a_2 & a_5 & a_7 & a_8 & a_{11} & a_{13} & a_{14} & a_{16} & a_{17} & a_{18} & a_{21} & a_{23} & a_{24} & a_{26} & a_{27} & a_{28} & a_{30} & a_{31} & a_{32} & a_{33} \\
x_3 & a_3 & a_6 & a_8 & a_9 & a_{12} & a_{14} & a_{15} & a_{17} & a_{18} & a_{19} & a_{22} & a_{24} & a_{25} & a_{27} & a_{28} & a_{29} & a_{31} & a_{32} & a_{33} & a_{34} \\
x_1^2 & a_4 & a_{10} & a_{11} & a_{12} & a_{20} & a_{21} & a_{22} & a_{23} & a_{24} & a_{25} & 0 & 0 & 0 & 0 & 0 & 0 & 0 & 0 & 0 & 0 \\
x_1x_2 & a_5 & a_{11} & a_{13} & a_{14} & a_{21} & a_{23} & a_{24} & a_{26} & a_{27} & a_{28} & 0 & 0 & 0 & 0 & 0 & 0 & 0 & 0 & 0 & 0 \\
x_1x_3 & a_6 & a_{12} & a_{14} & a_{15} & a_{22} & a_{24} & a_{25} & a_{27} & a_{28} & a_{29} & 0 & 0 & 0 & 0 & 0 & 0 & 0 & 0 & 0 & 0 \\
x_2^2 & a_7 & a_{13} & a_{16} & a_{17} & a_{23} & a_{26} & a_{27} & a_{30} & a_{31} & a_{32} & 0 & 0 & 0 & 0 & 0 & 0 & 0 & 0 & 0 & 0 \\
x_2x_3 & a_8 & a_{14} & a_{17} & a_{18} & a_{24} & a_{27} & a_{28} & a_{31} & a_{32} & a_{33} & 0 & 0 & 0 & 0 & 0 & 0 & 0 & 0 & 0 & 0 \\
x_3^2 & a_9 & a_{15} & a_{18} & a_{19} & a_{25} & a_{28} & a_{29} & a_{32} & a_{33} & a_{34} & 0 & 0 & 0 & 0 & 0 & 0 & 0 & 0 & 0 & 0 \\
x_1^3 & a_{10} & a_{20} & a_{21} & a_{22} & 0 & 0 & 0 & 0 & 0 & 0 & 0 & 0 & 0 & 0 & 0 & 0 & 0 & 0 & 0 & 0 \\
x_1^2x_2 & a_{11} & a_{21} & a_{23} & a_{24} & 0 & 0 & 0 & 0 & 0 & 0 & 0 & 0 & 0 & 0 & 0 & 0 & 0 & 0 & 0 & 0 \\
x_1^2x_3 & a_{12} & a_{22} & a_{24} & a_{25} & 0 & 0 & 0 & 0 & 0 & 0 & 0 & 0 & 0 & 0 & 0 & 0 & 0 & 0 & 0 & 0 \\
x_1x_2^2 & a_{13} & a_{23} & a_{26} & a_{27} & 0 & 0 & 0 & 0 & 0 & 0 & 0 & 0 & 0 & 0 & 0 & 0 & 0 & 0 & 0 & 0 \\
x_1x_2x_3 & a_{14} & a_{24} & a_{27} & a_{28} & 0 & 0 & 0 & 0 & 0 & 0 & 0 & 0 & 0 & 0 & 0 & 0 & 0 & 0 & 0 & 0 \\
x_1x_3^2 & a_{15} & a_{25} & a_{28} & a_{29} & 0 & 0 & 0 & 0 & 0 & 0 & 0 & 0 & 0 & 0 & 0 & 0 & 0 & 0 & 0 & 0 \\
x_2^3 & a_{16} & a_{26} & a_{30} & a_{31} & 0 & 0 & 0 & 0 & 0 & 0 & 0 & 0 & 0 & 0 & 0 & 0 & 0 & 0 & 0 & 0 \\
x_2^2x_3 & a_{17} & a_{27} & a_{31} & a_{32} & 0 & 0 & 0 & 0 & 0 & 0 & 0 & 0 & 0 & 0 & 0 & 0 & 0 & 0 & 0 & 0 \\
x_2x_3^2 & a_{18} & a_{28} & a_{32} & a_{33} & 0 & 0 & 0 & 0 & 0 & 0 & 0 & 0 & 0 & 0 & 0 & 0 & 0 & 0 & 0 & 0 \\
x_3^3 & a_{19} & a_{29} & a_{33} & a_{34} & 0 & 0 & 0 & 0 & 0 & 0 & 0 & 0 & 0 & 0 & 0 & 0 & 0 & 0 & 0 & 0 \\
\end{array}
$}\]

Imposing that $\Delta_{Q_0}(3)=2$ we have that the matrix 
{\small{
\[
\begin{array}{r|cccc}
& 1 & x_1 & x_2 & x_3\\
\hline 
1 & 0   & a_1 & a_2 & a_3\\
x_1 & a_1 & a_4 & a_5 & a_6 \\
x_2 & a_2 & a_5 & a_7 & a_8 \\
x_3 & a_3 & a_6 & a_8 & a_9 \\
x_1^2 & a_4 & a_{10} & a_{11} & a_{12} \\
x_1x_2 & a_5 & a_{11} & a_{13} & a_{14}\\
x_1x_3 & a_6 & a_{12} & a_{14} & a_{15}\\
x_2^2 & a_7 & a_{13} & a_{16} & a_{17} \\
x_2x_3 & a_8 & a_{14} & a_{17} & a_{18}\\
x_3^2 & a_9 & a_{15} & a_{18} & a_{19}\\
x_1^3 & a_{10} & a_{20} & a_{21} & a_{22}\\
x_1^2x_2 & a_{11} & a_{21} & a_{23} & a_{24}\\
x_1^2x_3 & a_{12} & a_{22} & a_{24} & a_{25}\\
x_1x_2^2 & a_{13} & a_{23} & a_{26} & a_{27}\\
x_1x_2x_3 & a_{14} & a_{24} & a_{27} & a_{28}\\
x_1x_3^2 & a_{15} & a_{25} & a_{28} & a_{29} \\
x_2^3 & a_{16} & a_{26} & a_{30} & a_{31} \\
x_2^2x_3 & a_{17} & a_{27} & a_{31} & a_{32}\\
x_2x_3^2 & a_{18} & a_{28} & a_{32} & a_{33} \\
x_3^3 & a_{19} & a_{29} & a_{33} & a_{34} \\
\end{array}\]
}}
has two columns corresponding to degree 3 polynomials that are linearly independent. A possibility is $a_{21}=1$, $a_7=1$, and the rest of the coefficients 0 of the first two columns and the degree 3 part of the last column equal to zero. Note that so far $f$ is determined by coefficients satisfying this choice. We could continue with this process until we recover the coefficients of $f$ (with possibly other solutions having the same Hilbert function).

\end{example}

\bibliographystyle{alphaurl}
\bibliography{bibliography}

\newcommand{\etalchar}[1]{$^{#1}$}
\begin{thebibliography}{BCMT10}

\bibitem[BBM14]{MR3250539}
A.~Bernardi, J.~Brachat, and B.~Mourrain.
\newblock A comparison of different notions of ranks of symmetric tensors.
\newblock {\em Linear Algebra Appl.}, 460:205--230, 2014.
\newblock \href {https://doi.org/10.1016/j.laa.2014.07.036} {\path{doi:10.1016/j.laa.2014.07.036}}.

\bibitem[BCMT10]{Bernard}
J.~Brachat, P.~Comon, B.~Mourrain, and E.~Tsigaridas.
\newblock Symmetric tensor decomposition.
\newblock {\em Linear Algebra and its Applications}, 433(11--12):1851--1872, 2010.
\newblock \href {https://doi.org/10.1016/j.laa.2010.06.046} {\path{doi:10.1016/j.laa.2010.06.046}}.

\bibitem[BGI11]{BERNARDI201134}
A.~Bernardi, A.~Gimigliano, and M.~Id{\`a}.
\newblock Computing symmetric rank for symmetric tensors.
\newblock {\em Journal of Symbolic Computation}, 46(1):34--53, 2011.
\newblock URL: \url{https://www.sciencedirect.com/science/article/pii/S0747717110001240}, \href {https://doi.org/10.1016/j.jsc.2010.08.001} {\path{doi:10.1016/j.jsc.2010.08.001}}.

\bibitem[BJM{\etalchar{+}}18]{BJPR}
A.~Bernardi, J.~Jelisiejew, P.~M. Marques, et~al.
\newblock On polynomials with given hilbert function and applications.
\newblock {\em Collectanea Mathematica}, 69:39--64, 2018.
\newblock \href {https://doi.org/10.1007/s13348-016-0190-2} {\path{doi:10.1007/s13348-016-0190-2}}.

\bibitem[BOT24]{GAD}
A.~Bernardi, A.~Oneto, and D.~Taufer.
\newblock On schemes evinced by generalized additive decompositions and their regularity.
\newblock {\em Journal de Math{\'e}matiques Pures et Appliqu{\'e}es}, 188:446--469, 2024.
\newblock \href {https://doi.org/10.1016/j.matpur.2024.06.007} {\path{doi:10.1016/j.matpur.2024.06.007}}.

\bibitem[Bou98]{Bourbaki}
N.~Bourbaki.
\newblock {\em Algebra I: Chapters 1-3}.
\newblock Actualit{\'e}s scientifiques et industrielles. Springer, 1998.
\newblock URL: \url{https://books.google.it/books?id=STS9aZ6F204C}.

\bibitem[BR13]{BR13}
A.~Bernardi and K.~Ranestad.
\newblock On the cactus rank of cubics forms.
\newblock {\em Journal of Symbolic Computation}, 50:291--297, 2013.

\bibitem[BT18]{Alessandra}
A.~Bernardi and D.~Taufer.
\newblock Waring, tangential and cactus decompositions.
\newblock {\em Journal de Math{\'e}matiques Pures et Appliqu{\'e}es}, 2018.

\bibitem[EM96]{BernardBook}
M.~Elkadi and B.~Mourrain.
\newblock {\em Introduction {\`a} la r{\'e}solution des syst{\`e}mes polynomiaux}.
\newblock Springer-Verlag Berlin Heidelberg, 1996.

\bibitem[Iar94]{Iarrobino}
A.~Iarrobino.
\newblock Associated graded algebra of a gorenstein artin algebra.
\newblock {\em Memoirs of the American Mathematical Society}, 107(514), 1994.

\bibitem[IK99]{IK}
A.~Iarrobino and V.~Kanev.
\newblock {\em Power Sums, Gorenstein Algebras and Determinantal Loci}, volume 1721 of {\em Lecture Notes in Mathematics}.
\newblock Springer-Verlag, 1999.

\bibitem[Jel22]{Joachim}
J.~Jelisiejew.
\newblock Hilbert schemes of points and applications.
\newblock 2022.
\newblock arXiv preprint.

\bibitem[Lan17]{MR3729273}
J.~M. Landsberg.
\newblock {\em Geometry and complexity theory}, volume 169 of {\em Cambridge Studies in Advanced Mathematics}.
\newblock Cambridge University Press, Cambridge, 2017.
\newblock \href {https://doi.org/10.1017/9781108183192} {\path{doi:10.1017/9781108183192}}.

\bibitem[Mou99]{NormalForm}
Bernard Mourrain.
\newblock A new criterion for normal form algorithms.
\newblock In {\em Proceedings of the 13th International Symposium on Applied Algebra, Algebraic Algorithms and Error-Correcting Codes}, AAECC-13, page 430–443, Berlin, Heidelberg, 1999. Springer-Verlag.

\bibitem[Mou18]{Polyexp}
B.~Mourrain.
\newblock Polynomial--exponential decomposition from moments.
\newblock {\em Foundations of Computational Mathematics}, pages 1435--1492, 2018.

\end{thebibliography}
\end{document}